\documentclass[11pt,reqno]{amsart}

\usepackage[dvipsnames]{xcolor}
\usepackage{tikz}
\usepackage{stmaryrd}
\usetikzlibrary{matrix,arrows,decorations.pathmorphing}
\usepackage{graphicx}

\usepackage{amssymb}
 \usepackage[stable]{footmisc}
\usepackage{amsmath,mathtools}
\usepackage{fullpage}
\usepackage{caption}
\usepackage{amsthm}
\usepackage{thmtools}
\usepackage{blkarray}
\usepackage{multirow}
\usepackage{picinpar} 
\usepackage{tikz-cd}
\usepackage{color}
\usepackage{verbatim}
\usepackage{hyperref}
\hypersetup{colorlinks=true, linkcolor=red, citecolor = OliveGreen, urlcolor = OliveGreen}
\usepackage{amssymb}
\usepackage{amsmath}
\usepackage{enumitem,colonequals}
\usepackage{setspace}
\usepackage{color}
\usepackage[all]{xy}
\usepackage{comment}
\usepackage{booktabs,array}
\usepackage[margin=1.in]{geometry}
\usepackage{tikz-cd}

\newcommand{\mZ}{\mathbb{Z}}

\newcommand{\mQ}{\mathbb{Q}}

\newcommand{\mA}{\mathbb{A}}

\newcommand{\mP}{\mathbb{P}}

\newcommand{\cC}{\mathcal{C}}

\newcommand{\cO}{\mathcal{O}}

\usepackage[OT2,T1]{fontenc}
\DeclareSymbolFont{cyrletters}{OT2}{wncyr}{m}{n}
\DeclareMathSymbol{\Sha}{\mathalpha}{cyrletters}{"58}
\DeclareMathSymbol{\Sha}{\mathalpha}{cyrletters}{"58}

\newcommand{\brk}[1]{ \left\lbrace #1 \right\rbrace }
\newcommand{\pwr}[1]{ \left( #1 \right) }
\newcommand{\zZ}[1]{ \mathbb{Z}/#1\mathbb{Z}}

\numberwithin{equation}{subsection}
\newtheorem{thmx}{Theorem}

\newtheorem{corx}[thmx]{Corollary}

\numberwithin{equation}{subsection}

\newtheorem{theorem}[subsection]{Theorem}

\newtheorem{lemma}[subsection]{Lemma}

\newtheorem{prop}[subsection]{Proposition}
\newtheorem{proposition}[subsection]{Proposition}

\theoremstyle{definition}
\newtheorem{defn}[subsection]{Definition}
\newtheorem{example}[subsection]{Example}
\newtheorem{question}[subsection]{Question}

\theoremstyle{remark}

\newtheorem{remark}[subsection]{Remark}

\numberwithin{equation}{section} \numberwithin{figure}{section}

\DeclareMathOperator{\Gal}{Gal} 
\DeclareMathOperator{\Aut}{Aut}

\DeclareMathOperator{\im}{Im}

\DeclareMathOperator{\lcm}{lcm}

\DeclareMathOperator{\GL}{GL}
\DeclareMathOperator{\gd}{gd}

\newcommand{\SL}{\textrm{SL}}

\newcommand{\Qbar}{\overline{\QQ}}

\newcommand{\cdef}[1]{\textsf{\textsf{#1}}}

\newcommand\FF{\mathbb{F}}

\newcommand\ZZ{\mathbb{Z}}
\newcommand\Z{\mathbb{Z}}
\newcommand\NN{\mathbb{N}}
\newcommand\QQ{\mathbb{Q}}

\newcommand\OO{\mathcal{O}}

\renewcommand{\leq}{\leqslant}

\renewcommand{\geq}{\geqslant}

\DeclareMathOperator{\ab}{ab}

\DeclareMathOperator{\GSp}{GSp}

\newcommand{\Q}{\mathbb Q}
\renewcommand{\NN}{\mathbb N}

\newcommand{\Rats}{\mathbb{Q}}

\newcommand{\arrow}{\longrightarrow}

\newcommand{\GQ}{\Gal(\overline{\Rats}/\Rats)}

\renewcommand{\Im}{{\rm Im}\,}

\makeatletter
\@namedef{subjclassname@1991}{\emph{2020} Mathematics Subject Classification}
\makeatother

\begin{document}
\title{Towards a classification of entanglements of Galois representations attached to elliptic curves}

\author{Harris B. Daniels} 
\address{Harris B. Daniels \\
	Department of Mathematics \\ 
	Amherst College \\
	Amherst, MA 01002
    USA}
\email{hdaniels@amherst.edu}

\author{\'Alvaro Lozano-Robledo}
\address{\'Alvaro Lozano-Robledo \\
	Department of Mathematics \\
	University of Connecticut \\
	Storrs, CT 06269
	USA}
\email{alvaro.lozano-robledo@uconn.edu}

\author{Jackson S. Morrow} 
\address{Jackson S. Morrow \\
	Centre de Recherche de Math\'ematiques, Universit\'e de Montr\'eal\\
    Montr\'eal, Qu\'ebec H3T 1J4\\
    CAN}
\email{jmorrow4692@gmail.com}


\date{\today}

\begin{abstract}
Let $E/\mathbb{Q}$ be an elliptic curve, let $\overline{\mathbb{Q}}$ be a fixed algebraic closure of $\mathbb{Q}$, and let $G_{\mathbb{Q}}=\text{Gal}(\overline{\mathbb{Q}}/\mathbb{Q})$ be the absolute Galois group of $\mathbb{Q}$. 
The action of $G_{\mathbb{Q}}$ on the adelic Tate module of $E$ induces the adelic Galois representation $\rho_E\colon G_{\mathbb{Q}} \to \text{GL}(2,\widehat{\mathbb{Z}}).$

The goal of this paper is to explain how the image of $\rho_E$ can be smaller than expected. To this end, we offer a group theoretic categorization of different ways in which an entanglement between division fields can be explained and prove several results on elliptic curves (and more generally, principally polarized abelian varieties) over $\mathbb{Q}$ where the entanglement occurs over an abelian extension. 
\end{abstract}
\maketitle

\section{\bf Introduction}
\label{sec:intro}

This article is concerned with the classification of (adelic) Galois representations attached to elliptic curves and, in particular, this paper is an attempt to systematically classify reasons that explain when the adelic image is not as large as possible. More concretely, let $E/\Q$ be an elliptic curve, let $\Qbar$ be a fixed algebraic closure of $\Q$, and let $G_\Q=\Gal(\Qbar/\Q)$ be the absolute Galois group of $\Q$. Then $G_\Q$ acts on $E[n]$, the $n$-torsion subgroup of $E(\Qbar)$, and induces a Galois representation
$$\rho_{E,n}\colon \GQ \to \Aut(E[n])\cong \GL(2,\Z/n\Z).$$
More generally, $\GQ$ acts on the Tate module $T(E)=\varprojlim E[n]$ of the elliptic curve $E$, and induces an adelic Galois representation
$$\rho_E\colon \GQ \to \Aut(T(E))\cong \GL(2,\widehat{\Z}).$$
This paper is then aimed at understanding the possibilities for the image $H_E$ of $\rho_E$ (a problem which is sometimes called Mazur's Program B, after \cite{mazurprogramb}). 

If $E$ has complex multiplication (CM), then $H_E$ fails to be open in $\GL(2,\widehat{\Z})$ 
and is well understood (see \cite{lombardo,bourdon-clark,lozanoCMreps}). Otherwise, if $E$ does not have CM, then Serre \cite{serre:OpenImageThm} showed that $H_E$ is an open subgroup of $\GL(2,\widehat{\Z})$. Let $\iota_E$ be the index $[\GL(2,\widehat{\Z}):H_E]$. For elliptic curves over $\Q$, it is always the case that $\iota_E\geq 2$. The reason for this is either because $\rho_{E,2}$ is not surjective or due to the existence of an explained non-trivial quadratic intersection of the division field  $\Q(E[2])$ and $\Q(\zeta_n)\subseteq \Q(E[n])$, for some $n>2$ (see Section \ref{subsec:serre_entangles}). In fact, $\iota_E$ is always even.

The goal of this paper is to understand the different scenarios that force $\iota_E$ to be larger than two. 
There are two phenomena that explain a larger-than-expected index, which we call a vertical Galois collapsing and a horizontal Galois entanglement, respectively. 
Here, we will describe these phenomena for non-CM curves $E/\Q$ and will refer the reader to Subsection \ref{sec-CM-entanglements} for these definitions in the CM setting.

\begin{itemize}
	\item For a prime $\ell$, we say that the $\ell$-adic reprsentation has a \cdef{vertical Galois collapsing} if the projection $\rho_{E,\ell^{\infty}}\colon G_{\Q} \to \GL(2,\widehat{\Z})\to \GL(2,\Z_\ell)$, which corresponds to the action of Galois on the $\ell$-adic Tate module, fails to be surjective. Vertical tanglements are (mostly) well-understood. If $\rho_{E,\ell^\infty}$ is not surjective, then either $\ell=2$ or $3$, or the mod-$\ell$ representation $\rho_{E,\ell}$ is not surjective (by \cite[IV-23, Lemma 3]{serreabelianl-adic}). If $\rho_{E,2^\infty}$ is not surjective, then $\rho_{E,8}$ is not surjective (see  \cite{dok-dok-surjectivity}), and if $\rho_{E,3^\infty}$ is not surjective, then $\rho_{E,9}$ is not surjective (see \cite{elkies3adic}). The conjectural list of all non-surjective images modulo $\ell$ can be found in \cite{zywinapossible}. 
	There is also work of Rouse--Zureick-Brown \cite{rouse2014elliptic} and more recent work of Rouse--Sutherland--Zureick-Brown \cite{RouseSZB:elladic} where the authors (conjecturally) classify all the possible $\ell$-adic images of Galois for elliptic curves $E/\mQ$ for all primes $\ell$. The classification of $\ell$-adic images in the CM case was described in \cite{lozanoCMreps}. 
	
	\item We say that an elliptic curve $E$ has a \cdef{horizontal Galois entanglement} if there exists a composite number $n\geq 6$ such that the projection $\rho_{E,n}\colon G_{\Q}\to \GL(2,\widehat{\Z})\to \GL(2,\Z/n\Z)$, which corresponds to the action of Galois on the $n$-torsion, fails to be surjective, and the entanglement is not explained by a product of vertical tanglements, i.e., 
	$$\rho_{E,n}(G_\Q) \subsetneq \prod_{\ell^a || n} \rho_{E,\ell^a}(G_\Q).$$
	In other words, there are integers $a,b\geq 2$, and at least one prime $q$ that divides $b$ but not $a$, such that $\Q(E[d])\subsetneq \Q(E[a])\cap \Q(E[b])$, where $d=\gcd(a,b)$.
\end{itemize}

In this work, we will primarily be interested in horizontal entanglements, and more concretely on abelian horizontal entanglements (see Section \ref{sec:classificationexplained} for the precise definition). 
To study these horizontal entanglements, we need to analyze the mod-$n$ image of the Galois representation $\rho_{E,n}$ for a composite integer $n\geq 6$. 
In \cite{danielsMorrow:Groupentanglements}, the first and last author describe how one can detect when an elliptic curve $E/\Q$ has a non-trivial $(a,b)$-entanglement, which corresponds to when the intersection of the $a$ and $b$-division fields is an extension of $\mQ$ that is not isomorphic to the $d$-division field, in terms of group theoretic data attached to the image of $\rho_{E,ab}$.  

\subsection{Main contributions}
Our first contribution to the study of entanglements is a categorization of the ways in which an abelian entanglement (Definition \ref{defn:abelian-type}) can occur. 
More precisely, we define several classes of abelian entanglements:~Weil, discriminant, CM, and fake CM entanglements; see Section \ref{sec:classificationexplained} for details. 
Our first result shows that, except for a finite number of isomorphism classes of elliptic curves, the abelian entanglements between two prime division fields are completely explained by Weil, discriminant, CM, and fake CM entanglements.

\begin{thmx}\label{thmx:main0}
	Let $E/\Q$ be an elliptic curve, and let $p$ and $q$ be distinct primes such that $E$ has an abelian $(p,q)$-entanglement of type $S$ (Definition \ref{defn:abelian-type}), for some finite abelian group $S$. Then, there is a finite set $J\subseteq \Q$, such that if $j(E)\not\in J$ and the entanglement is not Weil, discriminant, CM, or fake CM, then $S=\Z/3\Z$ and $(p,q)=(2,7)$ and $j(E)$ belongs to one of the following three explicit one-parameter families of $j$-invariants (which appear in Section 8.1 of \cite{danielsMorrow:Groupentanglements}):
	\begin{align*}
	j_1(t) &:= \frac{(t^{2} + t + 1)^{3}(t^{6} + 5t^{5} + 12t^{4} + 9t^{3} + 2t^{2} + t + 1)P_1(t)^{3}}{t^{14}(t + 1)^{14}(t^{3} + 2t^{2} - t - 1)^{2}}\\
	j_2(t) &:=  \frac{7^{4}(t^{2} + t + 1)^{3}(9t^{6} + 39t^{5} + 64t^{4} + 23t^{3} + 4t^{2} + 15t + 9)P_2(t)^{3}}{(t^{3} + t^{2} - 2t - 1)^{14}(t^{3} + 8t^{2} + 5t - 1)^{2}}\\
	j_3(t) &:= \frac{(t^{2} - t + 1)^{3}(t^{6} - 5t^{5} + 12t^{4} - 9t^{3} + 2t^{2} - t + 1)P_3(t)^{3}}{(t - 1)^{2}t^{2}(t^{3} - 2t^{2} - t + 1)^{14}}
	\end{align*}
	where
	\begin{align*}
P_1(t) &= t^{12} + 8t^{11} + 25t^{10} + 34t^{9} + 6t^{8} - 30t^{7} - 17t^{6} + 6t^{5} - 4t^{3} + 3t^{2} + 4t + 1,\\
P_2(t) &= \scalebox{0.95}{$t^{12} + 18t^{11} + 131t^{10} + 480t^{9} + 1032t^{8} + 1242t^{7} + 805t^{6} + 306t^{5} + 132t^{4} + 60t^{3} - t^{2} - 6t + 1$},\\
P_3(t) &= \scalebox{0.87}{$t^{12} - 8t^{11} + 265t^{10} - 1474t^{9} + 5046t^{8} - 10050t^{7} + 11263t^{6} - 7206t^{5} + 2880t^{4} - 956t^{3} + 243t^{2} - 4t + 1$}.
\end{align*}
Moreover, the set $J$ contains $\{-5^2/2, -5^2\cdot 241^3/2^3,-5\cdot 29^3/2^5,5\cdot 211^3/2^{15} \}$ and the (at most finitely many) $j$-invariants of elliptic curves with a $13$-isogeny and $\Gal(\Q(E[2])/\Q)\cong \Z/3\Z$, such that $\Q(E[2])\subseteq \Q(E[13])$. 
%
%
\end{thmx} 

\begin{corx}\label{cor-main0}
	Let $E/\Q$, $p$, and $q$ be as in Theorem \ref{thmx:main0}, and assume a positive answer to Serre's uniformity question. Then, there is a finite set $J'\subseteq \Q$, such that if $j(E)\not\in J'$ and the entanglement is not Weil, Serre, or CM, then:
	\begin{enumerate}
		\item $S=\Z/3\Z$ and $(p,q)=(2,7)$ and $j(E)$ belongs to one of the three families described in Theorem \ref{thmx:main0},
		\item $S=\Z/2\Z$ and $(p,q)=(2,q)$ with $q\in \{3,5,7,13\}$, and $j(E)$ belongs to one of the explicit one-parameter families of $j$-invariants with discriminant entanglements which appear in Section 8.1 of \cite{danielsMorrow:Groupentanglements}, 
		\item  $S=\Z/2\Z$ and $(p,q)=(3,5)$, and $j(E)$ belongs to one of the following two explicit one-parameter families of $j$-invariants which appear in Section 8.1 of \cite{danielsMorrow:Groupentanglements}:
				\begin{align*}
			j_4(t) &:= \frac{2^{12}P_4(t)^{3}}{(t - 1)^{15}(t + 1)^{15}(t^{2} - 4t - 1)^{3}}\\
			j_5(t) &:=  \frac{2^{12}P_5(t)^{3}}{(t - 1)^{15}(t + 1)^{15}(t^{2} - 4t - 1)^{3}}
			\end{align*}
			where
			\begin{align*}
		P_4(t) &= t^{12} - 9t^{11} + 39t^{10} - 75t^{9} + 75t^{8} - 114t^{7} + 26t^{6} + 114t^{5} + 75t^{4} + 75t^{3} + 39t^{2} + 9t + 1,\\
		P_5(t) &=  \scalebox{0.88}{$211t^{12} - 189t^{11} - 501t^{10} - 135t^{9} + 345t^{8} + 966t^{7} + 146t^{6} - 966t^{5} + 345t^{4} + 135t^{3}- 501t^{2} + 189t + 211$}.
		\end{align*}
		\item The set $J'$ contains all the $j$-invariants that correspond to non-cuspidal rational points on modular curves parametrizing elliptic curves such that they have non-surjective image at two primes $p,q\leq 37$, and at least one of the two is a contained in a normalizer of Cartan image, except for those infinite families mentioned above in (2) and (3). Further, $J\subsetneq J'$ as, for example, $j=11^3/2^3$ and $-29^3\cdot 41^3/2^{15}$ are in $J'$ but not in $J$. 
	\end{enumerate} 
\end{corx}
 
The proofs of Theorem \ref{thmx:main0} and Corollary \ref{cor-main0} are given in Section \ref{sec-other}. Each of $J$ and $J'$ corresponds to a collection of non-cuspidal, non-CM  $\mQ$-points on a finite number of higher genus modular curves, and therefore, it is difficult to explicitly determine it. While we make no attempt to do so in this work, it would be interesting to precisely compute $J$ and $J'$, and hence make the statement of the theorems more precise. 

In the remainder of the work, we study infinite families of Weil and Serre entanglements, describe how CM entanglements can be used to further analyze the adelic image of an elliptic curve with CM, and illustrate how group theoretic definitions of entanglements and our categorization of explained entanglements can be extended to principally polarized abelian varieties over $\Q$.

First, we prove results concerning infinite families of certain Weil entanglements (Definition \ref{defn:weil-type}).  
We exhibit several infinite families of elliptic curves over $\Q$ which have Weil $(a,b)$-entanglements for various tuples $(a,b)$. 
For example, we prove the following result, which is expanded on in Section \ref{sec:explainedentanglements}.

\begin{thmx}\label{thmx:main1}
There are infinitely many $\overline{\QQ}$-isomorphism classes of elliptic curves $E$ over $\Q$ satisfying one of the following conditions:
\begin{enumerate}
\item a Weil $(3,n)$-entanglement of type $\ZZ/2\ZZ$ where $3\nmid n$,
\item a Weil $(5,n)$-entanglement of type $\Z/4\Z$ where $5 \nmid n$,
\item a Weil $(7,n)$-entanglement of type $\Z/6\Z$ where $7 \nmid n$,
\item a Weil $(m,n)$-entanglement of type $\ZZ/2\ZZ$ where  $m\in \{4,6,9\}$ and $n\geq 3$ with $\gcd(m,n)\leq 2$.
\item a Weil $(m,\gcd(4|\Delta_E|,n))$-entanglement of type $\ZZ/2\ZZ\times \ZZ/2\ZZ$, where $m\in \{8,10,12\}$ and $n\geq 3$ with $\gcd(m,n)\leq 2$.
\end{enumerate}
\end{thmx}

In Section \ref{sec:explainedCM}, we shift our focus to entanglements of elliptic curves $E/\Q$ with complex multiplication. 
By work of the second author \cite[Theorem 1.2]{lozanoCMreps}, there is a compatible system of bases for $E[n]$ such that the image of $\rho_{E}$ is contained inside an extension $\mathcal{N}_{\delta,\phi}(\widehat{\Z})$ of a certain Cartan subgroup (see Definition \ref{defn-CMimage}) and the index of the image inside of this group divides $|\OO_{K,f}^\times|$. 
When the elliptic curve $E/\Q$ has CM, entanglements between division fields become much more prevalent. 
Indeed, if $E$ has CM by an order $\mathcal{O}_{K,f}$ of conductor $f\geq 1$ of an imaginary quadratic field $K$ and $n\geq 3$, then $K\subseteq \Q(E[n])$ by \cite[Lemma 3.15]{BCSTorPointsOnCM}, and so such elliptic curves will always have an entanglement between their $a$ and $b$ division fields once $a,b\geq 3$. 

We leverage this information to determine the index of the adelic representation associated to an elliptic curve over $\Q$ with CM, and we are able to precisely determine what this index is in several settings.

\begin{thmx}\label{thmx:main2}
Let $E/\Q$ be an elliptic curve with CM by an order $\OO_{K,f}$ of conductor $f\geq 1$ in an imaginary quadratic field $K$ with $\Delta_K\neq -4$ and $j(E)\neq 0$. 
For a choice of compatible bases of $E[n]$ for each $n\geq 2$, the index of the image of $\rho_{E}$ in $\mathcal{N}_{\delta,\phi}(\widehat{\mZ})$ is $2$.
\end{thmx}

\begin{remark}
When the $j$-invariant of $E$ is either 0 or 1728, we are not able to exactly determine the index of the image of $\rho_{E}$ in $\mathcal{N}_{\delta,\phi}(\widehat{\mZ})$, but we can exclude the possibility of certain indices occurring. Our work provides a different proof of certain cases of a more general result of Campagna and Pengo \cite{CampagnaAndPengo}, whose work we were not aware of until after our paper was finished. In their work (\cite[Corollary 4.6]{CampagnaAndPengo}) they show that for an elliptic curve over $\Q$ with CM by an order $\mathcal{O}_{K,f}$ of an imaginary quadratic field, the index of the image of $\rho_E$ in $\mathcal{N}_{\delta,\phi}(\widehat{\mZ})$ is $|\mathcal{O}_{K,f}^\times|$. Moreover, in \cite[Theorem 5.5]{CampagnaAndPengo}, they analyze the level where the entanglement happens. 
\end{remark}

\begin{remark}
We also mention another work of Campagna and Pengo \cite{CampagnaPengo:HowBig} which studies how large the image of $\rho_E(G_{\Q})$ is inside of a subgroup $\mathcal{G(E/\Q)}$ of $\Aut(T(E))$ when $E/\Q$ is a CM elliptic curve.  They provide a closed formula for the index $[\mathcal{G(E/\Q)} : \rho_{E}(G_{\Q})]$ using the classical theory of complex multiplication.  We refer the reader to \textit{loc.~cit.~}Section 2 for the definition of $\mathcal{G(E/\Q)}$. 
\end{remark}

Our final result concerns entanglements for principally polarized abelian varieties $A$ over $\Q$ of arbitrary dimension. 
The definition of entanglements in terms of intersection of division fields carries over for abelian varieties, however the group theoretic definitions require more care as the determinant of the mod $n$ image of Galois associated to $A$ over $\Q$ need not be surjective. 
In Section \ref{sec:explainedhigherdim}, we generalize a result of the authors concerning an infinite family of elliptic curves over $\Q$ with a Weil $(2,n)$-entanglement of type $\zZ{3}$ to the setting of principally polarized abelian varieties.  

\begin{thmx}\label{thmx:main3}
Let $\ell > 3$ be a prime number. 
Suppose that $\ell - 1 = 2e$ where $e = 2g+1$ is some odd integer. 
There exist infinitely many principally polarized abelian varieties $A/\mQ$ of dimension $g$ which have a Weil $(2,\ell)$-entanglement of type $\zZ{e}$.
\end{thmx}

In addition to these results, we also pose two question related to the behavior of Weil entanglements (see Questions \ref{question:explained} and \ref{question:boundsconductor}). 

\subsection{Related results}
Over the past five years, the study of entanglements has significantly developed and matured. 
The first results on entanglements can be traced to Brau--Jones \cite{brau3in2} and the last author \cite[Theorem 8.7]{morrow2017composite} where their results classified all elliptic curves $E/\Q$ with $(2,3)$-entanglement of non-abelian type. 
More recently, the first and second author \cite{danielsLR:coincidences} determined the elliptic curves $E/\Q$ and primes $p$ and $q$ such that $\mQ(E[p])\cap \mQ(\zeta_{q^k})$ is non-trivial and determined the degree of this intersection, and consequently, they describe all elliptic curves $E/\Q$ and integers $m,n$ such that the $m$-th and $n$-th division fields coincide.
In the setting of CM elliptic curves, work of Campagna--Pengo \cite{CampagnaAndPengo} studies when division fields become linearly disjoint.

Finally, there have been very recent works which attempt to systematically study entanglements. 
Jones--McMurdy \cite{JonesM:Nonabelianentanglements} have determined the genus zero modular curves and their $j$-maps where rational points on these modular curves correspond to elliptic curves over any number field $K$ with entanglements of non-abelian type. 
Also, the first and last author \cite{danielsMorrow:Groupentanglements} initiated a group theoretic perspective to studying entanglements and completely classified the infinite families of elliptic curves over $\Q$ which have an ``unexplained" $(p,q)$-entanglement where $p,q$ are distinct primes.

\subsection{Outline of paper}
In Section \ref{sec:grouptheoreticdefinitions}, we recall the group theoretic definition of explained entanglements established by the first and last author in \cite{danielsMorrow:Groupentanglements}. 
In Section \ref{sec:classificationexplained}, we offer a categorization of explained entanglements, and in particular, we define the notions of abelian, Weil, discriminant (and Serre), and CM entanglements.  
In Section \ref{sec-proofs_main_theorems}, we prove Theorem \ref{thmx:main0} and Corollary \ref{cor-main0}. 

The next two sections focus on proving results related to our categorization. In Section \ref{sec:explainedentanglements}, we construct several infinite families of elliptic curves over $\Q$ which have certain Weil entanglements using various methods such as division polynomials and isogeny-torsion graphs and prove Theorem \ref{thmx:main1}. 
In Section \ref{sec:explainedCM}, we analyze the adelic image of Galois associated to a CM elliptic curve and prove Theorem \ref{thmx:main2}. 
We extend the group theoretic notions of entanglements to principally polarized abelian varieties over $\Q$ of arbitrary dimension and prove Theorem \ref{thmx:main3} in Section \ref{sec:explainedhigherdim}. 

Finally, in Section \ref{sec:concludingremarks}, we discuss future avenues of study for explained entanglements. 

\subsection{Conventions}
We establish the following conventions throughout the paper.

\subsubsection*{Elliptic curves}
For a field $k$, we will use $E/k$ to denote an elliptic curve over $k$. 
For an element $d \in k^{\times}/(k^{\times})^2$, the twist of $E$ by $d$ will be denoted by $E^{(d)}$. 
Any particular elliptic curve over $\mQ$ mentioned in the paper will be given by LMFDB reference and a link to the corresponding LMFDB \cite{lmfdb} page when possible. 
We will abbreviate an elliptic curve $E/k$ having complex multiplication by an order $\cO_{K,f}$ in an imaginary quadratic field $K$ by saying that $E$ has CM by $\mathcal{O}_{K,f}$. 
We will use the notation $j_{K,f}$ to denote a $j$-invariant with CM by the order of conductor $f$ of $\cO_K$ where all other such $j$-invariants are conjugates of this one. 

\subsubsection*{Groups}
We set some notation for specific subgroups of $\GL(2,\zZ{\ell})$ for $\ell \geq 3$. Let $\ell{\text Cs}$ be the subgroup of diagonal matrices. 
Let $\epsilon = -1$ if $\ell \equiv 3 \pmod 4$ and otherwise let $\epsilon \geq 2$ be the 
smallest integer which is not a quadratic residue modulo $\ell$. 
Let $\ell{\text Cn}$ be the subgroup consisting of matrices of the form $\begin{psmallmatrix} a & b\epsilon \\ b & a \end{psmallmatrix}$ with $(a,b) \in \zZ{\ell}^2 \setminus \brk{(0,0)}$. 
Let $\ell{\text Ns}$ and $\ell{\text{Nn}}(\ell)$ be the normalizers of $\ell{\text Cs}$ and $\ell{\text Cn}$, respectively, in $\GL(2,\zZ{\ell})$. 
We have $[\ell{\text{Ns}} : \ell{\text{Cs}}] = 2$ and the non-identity coset of $\ell{\text{Cs}}(\ell)$ in $\ell{\text{Ns}}(\ell)$ is represented by $\begin{psmallmatrix} 0 & 1 \\ 1 & 0 \end{psmallmatrix}$. 
Similarly, $[\ell{\text{Nn}} : \ell{\text{Cn}}] = 2$ and the non-identity coset of $\ell{\text{Cn}}$ in $\ell{\text{Nn}}$ is represented by $\begin{psmallmatrix} 1 & 0 \\ 0 & -1 \end{psmallmatrix}$. Let $\ell{\text B}$ be the subgroup of upper triangular matrices in $\GL(2,\zZ{\ell})$. 
This notation was established by Sutherland in \cite{Sutherland} and is used in the LMFDB \cite{lmfdb}. We will also use Sutherland's notation for the less standard subgroups of level $\ell$ (and also for those of level $\ell=2$). 

\subsection{Comments on code}
All of the computations in this paper were performed using \texttt{Magma} \cite{Magma}. The code used to do the computations can be found at the following link. 
\begin{center}
{\color{black}{\url{https://github.com/jmorrow4692/ExplainedEntanglements}}}
\end{center}

\subsection{Acknowledgments}
We thank Shiva Chidambaram, Jeremy Rouse, and Andrew Sutherland for helpful correspondences. 
We are grateful to Lea Beneish, Abbey Bourdon, Garen Chiloyan, and Nathan Jones for many helpful comments on an earlier draft of this paper, and to Hanson Smith for conversations related to the phrase vertical collapsing. We would also like to thank Francesco Campagna and Riccardo Pengo for pointing out some inaccuracies in an earlier version of the paper and Zo\'e Yvon for alerting us to a needed correction in Definition \ref{defn:weil-type}. 
Finally, we would like to thank the referees for their useful comments and suggestions on how to improve the paper.

\section{\bf Group theoretic definition of entanglement}
\label{sec:grouptheoreticdefinitions}
In this section, we recall the group theoretic definitions for explained entanglements from  \cite{danielsMorrow:Groupentanglements}. 

\subsection*{Notation}
Let $G$ be a subgroup of $\GL(2,\Z/n\Z)$ for some $n\geq 2$ with surjective determinant, let $a <  b$ be divisors of $n$, let $c= \lcm(a,b)$, and let $d=\gcd(a,b)$. 
Let $\pi_c\colon \GL(2,\zZ{n})\to \GL(2,\zZ{c})$ denote the natural projection map, and set $G_c := \pi_c(G)$. 
Then, for each $e\in \{a,b,d\}$, we have the following reduction maps and normal subgroups of $G_c$
$$\pi_e\colon \GL(2,\zZ{c}) \to \GL(2,\zZ{e}), \quad  N_e := \ker(\pi_e) \cap G_c.$$

  We now recall one of the definitions for when $G$ represents an $(a,b)$-entanglement. 

\begin{defn}\label{def:representsentanglement}
With the notation above, we say that the group $G$ \cdef{represents a non-trivial $(a,b)$-entanglement} if 
\[
\langle N_a,N_b\rangle \subsetneq N_d.\]
The \cdef{type} of the entanglement is the isomorphism type of the group $N_d/\langle N_a, N_b\rangle$. 
\end{defn}

Below is a diagram summarizing the notation and definitions for group-theoretical entanglements.

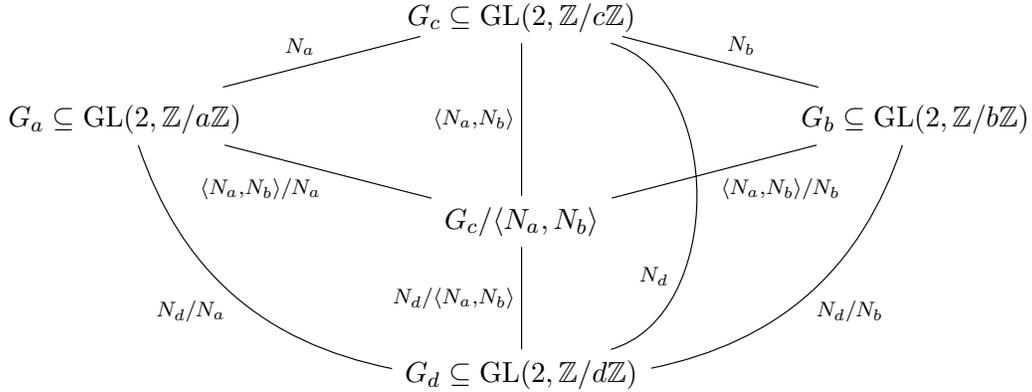
\begin{figure}[h!]
	\begin{tikzcd}
		&  & G_c\subseteq\GL(2,\mathbb{Z}/c\mathbb{Z}) \arrow[dd, "{\langle N_a,N_b\rangle}"', no head] \arrow[lld, "N_a"', no head, shift right=1] \arrow[rrd, "N_b", no head, shift left=1] \arrow[dddd, "N_d"', no head, crossing over, bend left=74, near end] &  &                                                                                                                                                            \\
		G_a\subseteq \GL(2,\mathbb{Z}/a\mathbb{Z}) \arrow[rrd, "{\langle N_a,N_b\rangle/N_a}"', no head] \arrow[rrddd, "N_d/N_a"', no head, bend right] &  &                                                                                                                                                                                                                                    &  & G_b\subseteq\GL(2,\mathbb{Z}/b\mathbb{Z}) \arrow[lld, "{\langle N_a,N_b\rangle/N_b}", no head] \arrow[llddd, "N_d/N_b", no head, bend left] \\
		&  & {G_c/\langle N_a,N_b\rangle} \arrow[dd, "{N_d/\langle N_a,N_b\rangle}"', no head]                                                                                                                                                  &  &                                                                                                                                                            \\
		&  &                                                                                                                                                                                                                                    &  &                                                                                                                                                            \\
		&  & G_d\subseteq\GL(2,\mathbb{Z}/d\mathbb{Z})                                                                                                                                                                            &  &                                                                                                                                                           
	\end{tikzcd}
\caption{A diagram describing group-theoretical entanglements. Lines denote maps, together with the subgroup that one has to quotient the group on top by to get to the group below.}
\label{fig-entangle-alt} 
\end{figure}

We now impose a maximality condition on a group representing a non-trivial entanglement. 

\begin{defn}\label{defn:primitiveentanglement}
Consider the set 
\[
\mathcal{T}_G = \{ ( (a,b), H) \, | \, G \text{ represents a non-trivial $(a,b)$-entanglement of type $H$} \}.
\] 
We define a relation on $\mathcal{T}_G$ by declaring that $ ((a_1,b_1), H_1) \leq ( (a_2,b_2), H_2 )$ if:
\begin{enumerate}
\item $H_1$ and $H_2$ are isomorphic and either $(a_2 \mid a_1$ and $b_2\mid b_1)$ or  $(b_2 \mid a_1$ and $a_2 \mid b_1)$, or
\item $H_1$ is isomorphic to a quotient of $H_2$ and either $(a_1 \mid a_2$ and $b_1 \mid b_2)$ or $(b_1 \mid a_2$ and $a_1 \mid b_2)$.
\end{enumerate}
We say the group $G$ represents a \cdef{primitive $(a,b)$-entanglement of type $H$} if $( (a,b), H )$ is the unique maximal element of $\mathcal{T}_G$ and $n = \lcm(a,b)$. 
\end{defn}

\begin{remark}
For a given $G$, if the set $\mathcal{T}_G$ has no maximal element with respect to this relation, then $G$ does not represent any primitive entanglement. 
\end{remark}


We now apply these definitions to elliptic curves. 
For $E/\mQ$ an elliptic curve and for an integer $n\geq 2$, let $G = \Im(\rho_{E,n}) \subseteq \GL(2,\zZ{n})$ denote the image of the Galois representation associated to the action on $E[n]$.  
Recall that the $n$-division field $\mQ(E[n])$ is the fixed field of $\overline{\mQ}$ by the kernel of the representation $\rho_{E,n}$, and so the Galois group of this number field is isomorphic to the image of the mod $n$ representation.

In \cite[Definition 1.1]{danielsMorrow:Groupentanglements}, the first and last author defined a notion of entanglement for elliptic curves, which we now rephrase in terms of the definitions of this paper.

\begin{defn}\label{def:ellipticentanglement}
We say that an elliptic curve $E/\mQ$ has a \cdef{non-trivial $(a,b)$-entanglement of type $T$} if for some $n\geq 2$ and proper divisors $a < b$, the mod $n$ image of Galois $\,\Im(\rho_{E,n})$ represents a non-trivial $(a,b)$-entanglement of type $T$.
\end{defn}

\begin{remark}
We note that the definitions of horizontal entanglement from Section \ref{sec:intro} and non-trivial $(a,b)$-entanglement coincide. 
\end{remark}


\section{\bf Classification of abelian entanglements}
\label{sec:classificationexplained}
In this work, we are interested in the study and classification of abelian entanglements. 
To this end, we offer the following classification. 
Our guiding principle is that an abelian entanglement will happen when there is some underlying arithmetic reason for it to.  
Below, we will highlight several arithmetic constraints which elucidate the existence of the abelian entanglement. 

\subsection{\bf Abelian and Weil entanglements}
To begin, we define notions of abelian entanglements and Weil-type entanglements.

\begin{defn}\label{defn:abelian-type}
	Let $E/\mQ$ be an elliptic curve, let $a<b$ be integers, let $d=\gcd(a,b)$, and suppose $S$ is a non-trivial finite abelian group. 
	Let $K=\Q(E[a])\cap \Q(E[b])$. We say that $E/\mQ$ has an abelian entanglement, or more precisely, that it has an \cdef{abelian $(a,b)$-entanglement of type $S$} if  $\Gal(K\cap \Q^\text{ab}/\Q(E[d]) \cap \Q^{\text{ab}})\cong S$.
	\end{defn}
	\begin{defn}\label{defn:weil-type}
Let $E/\mQ$ be an elliptic curve, let $a<b$ be integers, let $d=\gcd(a,b)$, and suppose $T$ is a non-trivial finite abelian group. 
Let $K_a=\mQ(E[a])\cap \Q^\text{ab}$, and similarly $K_b=\mQ(E[b])\cap \Q^\text{ab}$. We say that $E/\mQ$ has an entanglement of Weil-type, or more precisely, that it has a \cdef{Weil $(a,b)$-entanglement of type $T$} if 
$\Gal(K_a\cap \Q(\zeta_b)/\mQ(\zeta_d))\cong T$ or $\Gal(K_b\cap \Q(\zeta_a)/\mQ(\zeta_d))\cong T$.
	\end{defn}
	
	\begin{remark}\label{rem:Weilexplained}
	We note that an elliptic curve has a Weil $(a,b)$-entanglement if and only if it has an explained $(a,b)$-entanglement as in \cite[Definition 3.7]{danielsMorrow:Groupentanglements}. 
	Indeed, this follows from the discussion in Remark 4.4 of \textit{loc.~cit.}.
	\end{remark}


For a diagram summarizing the definition of abelian and Weil entanglements, we refer the reader to Figure \ref{fig-Weilentanglement}. 

\begin{figure}
\begin{tikzcd}
	&  & {\mathbb{Q}(E[c])} \arrow[lld, no head] \arrow[dd, no head] \arrow[rrd, no head] & &  \\
	{\mathbb{Q}(E[a])} \arrow[rrd, no head] \arrow[dd, no head] &  &   &  & {\mathbb{Q}(E[b])} \arrow[lld, no head] \arrow[dd, no head] \\
	&  & {K=(\mathbb{Q}(E[a])\cap \mathbb{Q}(E[b]))} \arrow[d, no head] \arrow[rd, no head] & &  \\
	\mathbb{Q}(\zeta_a) \arrow[rrdd, no head] &  & {\mathbb{Q}(E[d])} \arrow[d, no head]  & K\cap \mathbb{Q}^\text{ab} \arrow[d, no head] & \mathbb{Q}(\zeta_b) \arrow[ld, no head]                     \\
	&  &   {\mathbb{Q}(E[d])} \cap \Q^{\text{ab}}\arrow[d, no head] \arrow[ru, "S", no head]  & K\cap \mathbb{Q}(\zeta_b) \arrow[ld, "T", no head] &  \\
	&  & \mathbb{Q}(\zeta_d) \arrow[d, no head]   & &  \\
	&  & \mathbb{Q}  & &                                                            
\end{tikzcd}
\caption{An abelian $(a,b)$-entanglement of type $S$, and a Weil $(a,b)$-entanglement of type $T$, where $c=\lcm(a,b)$ and $d=\gcd(a,b)$.}
\label{fig-Weilentanglement}
\end{figure}
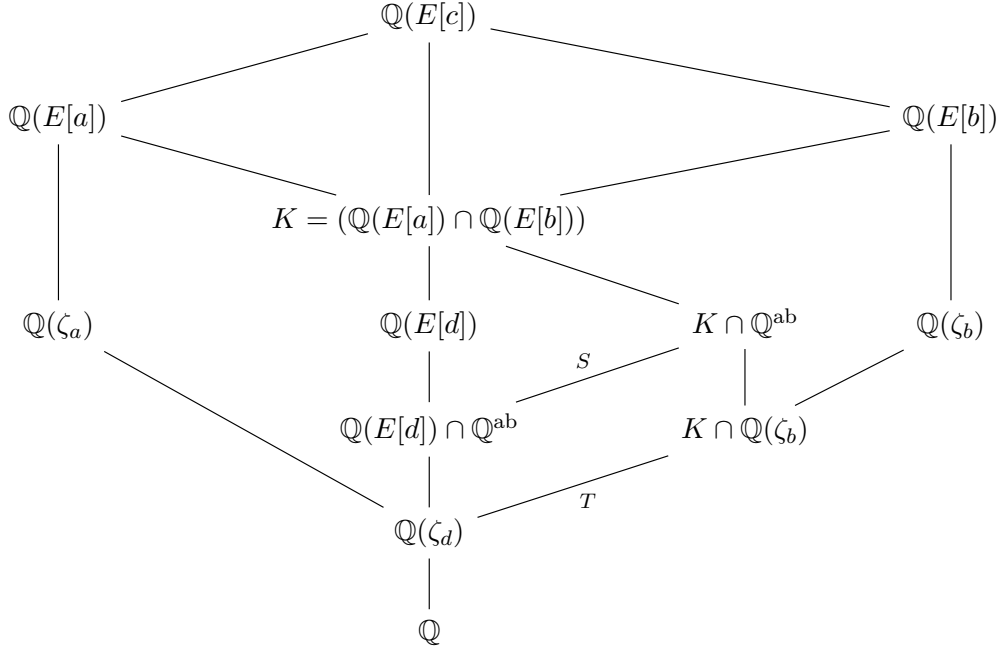

\begin{example}[An abelian entanglement that is not Weil]\label{ex-discriminant-abelian-ent}
This example shows that not all abelian entanglements are of Weil type. For example, let $E/\mQ$ be the elliptic curve with LMFDB label \href{https://www.lmfdb.org/EllipticCurve/Q/448/g/3}{\texttt{448.g3}}. Then, $\Q(E[2])=\Q(\sqrt{2})$ while $\Q(E[3])=\Q(E[6])=\Q(\sqrt{2},\sqrt{-3})$. In particular, $\Q(E[2])\cap \Q(E[3])=\Q(\sqrt{2})$ but $\Q(\sqrt{2})$ is not contained in $\Q(\zeta_2)$ nor $\Q(\zeta_3)$. Thus, here $E$ has an abelian $(2,3)$-entanglement of type $S=\Z/2\Z$ which does not arise as a Weil entanglement. 
\end{example}

\begin{example}[A Weil entanglement that is not abelian]\label{ex-discriminant-weil-ent}
	This example shows that not all Weil-type entanglements are of abelian type as defined in Definition \ref{defn:abelian-type}. For example, let $E/\mQ$ be the elliptic curve with LMFDB label \href{https://www.lmfdb.org/EllipticCurve/Q/21/a/1}{\texttt{21.a1}}. The discriminant of $E$ is $\Delta_E = 3\cdot 7^2$, and therefore $\Q(E[2])$ is entangled with $\Q(\sqrt{3})\subseteq \Q(\zeta_{12})\subseteq \Q(E[12])$ and we have a Weil $(2,12)$-entanglement. Note that $\Q(\sqrt{3})$ is not contained in $\Q(E[3])$ so there is no Weil $(2,3)$-entanglement. However, $d=\gcd(2,12)=2$, and $K\cap \Q^{\text{ab}}/\Q(E[d])\cap \Q^{\text{ab}}$ is the trivial extension $\Q(\sqrt{3})/\Q(\sqrt{3})$, so there is no abelian $(2,12)$-entanglement.
\end{example}

See Example \ref{ex-both-abelian-and-weil} for a family of elliptic curves that have an entanglement that is both of the abelian and Weil-type, simultaneously. 

We note that in the Examples \ref{ex-discriminant-abelian-ent} and \ref{ex-discriminant-weil-ent}, we have that the entanglement field is  $\Q(\sqrt{\Delta_E})$, which is a type of entanglement that we define in the following section in greater generality.

\subsection{\bf Discriminant entanglements}\label{subsec:disc_entangles} A \cdef{discriminant entanglement} is an entanglement between division fields $\Q(E[2])$ and $\Q(E[n])$ such that $\Q(\sqrt{\Delta_E})\subseteq \Q(E[2])\cap \Q(E[n])$. Whenever $\Delta_E$ is not a square in $\QQ^\times$, the field  $\QQ(\sqrt{\Delta_E})$ is the unique quadratic extension of $\QQ$ inside $\QQ(E[2])$. 
 If $n$ is even, $\Q(E[2])\subseteq \Q(E[n])$, so we refine the notion of discriminant entanglement as follows. 

\begin{defn}\label{defn:disc_entang}
	Let $E/\Q$ be an elliptic curve, and let $\Delta_E$ be the discriminant of $E$.
	\begin{enumerate}
		\item If $E/\Q$ has a square discriminant $\Delta_E\in (\QQ^\times)^2$, then $\Q(\sqrt{\Delta_E})/\Q$ is trivial, which means that $\Gal(\Q(E[2])/\Q)$ is either isomorphic to $\Z/3\Z$ or is trivial. In either case, the image of $\rho_{E,2}$ is smaller than $\GL(2,\Z/2\Z)$ and therefore $E/\Q$ has a vertical collapsing (see Section \ref{sec:intro}). We call this entanglement a \cdef{vertical discriminant collapsing}.
		\item If $\Delta_E$ is not a square and $n \geq 3$, then we let $m = \lcm(2,n)$, and $G_m = \Im (\rho_{E,m})$. If there exists a group $\mathcal{G} \subseteq \GL(2,\ZZ/m\ZZ)$ such that $G_m\subseteq \mathcal{G}$ is of index $2$, and there two index-$2$ subgroups $N_2$ and $N_n$ of $\mathcal{G}$ with the following properties:
\begin{itemize}
\item $N_2$ is the index-$2$ normal subgroup of $\mathcal{G}$ corresponding to $\Q( \sqrt{\Delta_E})$, in other words, $N_2$ is the unique index-$2$ normal subgroup of $\mathcal{G}$ such that $[\pi_2(\mathcal{G}):\pi_2(N_2)]=2$ as subgroups of $\GL(2,\Z/2\Z)$, where $\pi_2\colon \GL(2,\Z/m\Z)\to \GL(2,\Z/2\Z)$ is the natural reduction map, 
\item $N_2 \neq N_n$, and
\item $G_m \cap N_2 = G_m \cap N_n$,
\end{itemize}
then we say that $E$ has a \cdef{$(2,n)$-discriminant entanglement.}	
\end{enumerate}
\end{defn}

Before proceeding we make a remark and provide a few examples to illuminate Definition \ref{defn:disc_entang}, in particular Part (2). 

\begin{remark}
We now describe the idea behind part (2) of Definition \ref{defn:disc_entang}. 
With $\mathcal{G}$ as in part (2), the rational points on the modular curve $X_\mathcal{G}$ correspond to elliptic curves whose $m$-division fields have two \emph{distinct} quadratic fields. One of these quadratic fields is $\Q(\sqrt{\Delta_E})$ (the fixed field of $N_2$) and the other one is the fixed field of $N_n$, which is distinct from $\Q(\sqrt{\Delta_E})$ since $N_2 \neq N_n$ and quadratic since $[\mathcal{G}:N_n] =2$. 
The group $G_m$ is contained in this group $\mathcal{G}$ and the group $G_m$ the can be constructed by gluing $N_2$ and $N_n$ together.
At the level of fields, this corresponds to the fixed field of $N_2$ and $N_n$ coinciding. 
We note that it is not required for $\mathcal{G}$ to be an admissible group containing $-I$,  however we restrict to this case in order to describe the ideas precisely.  
\end{remark}

\begin{example}\label{ex:disc_ent_1}
Let $E$ be the elliptic curve with LMFDB label \texttt{35.a2}. Using \cite{RouseSZB:elladic} and \cite{lmfdb} we see that $\Im (\rho_{E,3^\infty})$ is exactly the group \texttt{3B.1.1} lifted to $\GL(2,\Z_3)$, and $\rho_{E,\ell^\infty}$ is surjective onto $\GL(2,\Z_\ell)$ otherwise.
Let $G_6$ be the mod-$6$ full preimage of the group \texttt{3B.1.1} in $\GL(2,\Z/6\Z)$ and let $\mathcal{G}$ be the group corresponding to the the preimage of \texttt{3B} but again inside of $\GL(2,\Z/6\Z)$. We note that $[\mathcal{G}:G_6]=2$ in $\GL(2,\Z/6\Z)$.

The normal subgroup 
$$N_2 = \left\langle
\begin{pmatrix}
2&5\\3&1
\end{pmatrix},
\begin{pmatrix}
2&3\\3&5
\end{pmatrix},
\begin{pmatrix}
2&3\\3&1
\end{pmatrix}
\right\rangle,$$
of $\mathcal{G}$ corresponds to the field $\Q(\sqrt{\Delta_E}) = \Q(\sqrt{-35})$. Next let
$$N_n := \left\langle
\begin{pmatrix}
2&1\\3&2
\end{pmatrix},
\begin{pmatrix}
4&5\\3&5
\end{pmatrix}
\right\rangle.$$
With this, one can verify that 
\begin{itemize}
\item $N_n \neq N_2$ and
\item $N_n \cap G_6 = N_2 \cap G_6$,
\end{itemize}
and hence we see that $E$ has a $(2,6)$-discriminant entanglement. We remark here that this is neither a Weil entanglement, as $\Q(\sqrt{-35})$ is not contained in $\Q(\zeta_6)$, nor an abelian entanglement (in the sense of Definition \ref{defn:abelian-type}) because $\lcm(2,6)=2$ and $\Q(E[2])\cap \Q(E[6]) \cap \Q^{\text{ab}}/(\Q(E[2])\cap \Q^{\text{ab}})$ is the trivial extension $\Q(\sqrt{-35})/\Q(\sqrt{-35})$.
\end{example}

\begin{remark}
We point out that any elliptic curve $E'/\Q$ whose mod-$6$ Galois representation has image conjugate to $\mathcal{G}\subsetneq \GL(2,\Z/6\Z)$ has a point $P$ of order 3 defined over $\Q$ or a quadratic field, call it $\Q(\sqrt{d})$. A group theoretic computation shows that $N_n$ is precisely the subgroup that fixes the subfield $\Q(\sqrt{d\Delta_{E'}})\subseteq \Q(E'[6])$ and $N_2$ fixes $\Q(\sqrt{\Delta_{E'}})$. 

Thus, if $E$ is an elliptic curve whose mod-$6$ Galois representation has image conjugate to $G_6\subsetneq\mathcal{G}\subsetneq \GL(2,\Z/6\Z)$, such as $E$ in Example \ref{ex:disc_ent_1},
then the entanglement we see is exactly $\Q(\sqrt{\Delta_{E'}}) = \Q(\sqrt{d\Delta_{E'}}).$ 
Notice that this would mean $$\Delta_{E'} \equiv d \Delta_{E'} \bmod (\QQ^\times)^2,$$ which forces $d \equiv 1$ modulo squares. 
That is, $d$ is a square, and so $P \in E(\Q)$. 
\end{remark}

Motivated by Example \ref{ex:disc_ent_1} and the previous remark, we provide one more example to show that part (1) of Definition \ref{defn:disc_entang} can be viewed in a similar light as in the case of part (2). 

\begin{example}
Let $E/\Q$ be the elliptic curve with Cremona label \texttt{392f1.}	 In this case the discriminant of $E$ is $\Delta_E = 784 = 2^4\cdot 7^2$ and $\rho_{E,\ell^\infty}$ is surjective everywhere except for $\ell=2$ where the image   is the preimage of \texttt{2Cn} in $\GL(2,\Z_\ell)$. Note that \texttt{2Cn} is the subgroup of $\GL(2,\Z/2\Z)$ of order 3. 
In this case, since $\Delta_E$ is a square, the normal subgroup $N_2$ does not exist (the subgroup fixing $\Q(\sqrt{\Delta_E})=\Q$ is the entire group, so not of index $2$), but we can still give a group-theoretic description of the entanglement as follows.
To see this, let $G_6 = \Im (\rho_{E,6})$ and $\mathcal{G} = \GL(2,\Z/6\Z)$. 
Then, if we let

$$N_a := \left\langle
\begin{pmatrix}
	5&1\\4&3
\end{pmatrix},
\begin{pmatrix}
	2&5\\5&1
\end{pmatrix},	
\begin{pmatrix}
	2&1\\1&1
\end{pmatrix},	
\begin{pmatrix}
    2&5\\5&0
\end{pmatrix}
\right\rangle$$
and
$$
N_b := \left\langle
\begin{pmatrix}2&1\\5&0\end{pmatrix},
\begin{pmatrix}2&5\\5&1\end{pmatrix},
\begin{pmatrix}2&1\\1&1\end{pmatrix}
\right\rangle$$
then the subgroup $N_a$ corresponds to $\Q(\sqrt{-3})\subseteq \Q(\zeta_3)$ while $N_b$ corresponds to $\Q(\sqrt{-3\Delta_E})$.
In $G_6$, we have that
$G_6 \cap N_a = G_6 \cap N_b$.

The group $N_a \cap G_6$ is exactly the normal subgroup of $G_6$ that fixes $\Q(\sqrt{-3})$. In other words, the group $G_6$ is characterized by the coincidence of the fields $\Q(\sqrt{-3})$ and $\Q(\sqrt{-3\Delta_E})$. This is the same as $\Q(\sqrt{\Delta_E}) = \Q.$ So, in this sense parts (1) and (2) are compatible. 

Before moving on, we also point out that there is nothing particularly special about $p=3$. We could do the same thing for any odd prime by letting $N_a$ be the normal subgroup that fixes $$\QQ\left(\sqrt{p^*}\right)\subseteq\QQ(\zeta_p)\subseteq\QQ(E[p])$$ where $p^* =(-1)^{\frac{p-1}{4}}p$. Then, in this case $p^*\Delta_E\equiv p^* \bmod (\QQ^\times)^2$ forces $\Delta_E \equiv 1 \bmod (\QQ^\times)^2.$
\end{example}

The most common of discriminant entanglements was described by Serre, and we define it next.

\subsubsection{\bf The Serre (or discriminant) entanglement}\label{subsec:serre_entangles} 
The \cdef{Serre entanglement} is a discriminant entanglement (as in Definition \ref{defn:disc_entang}) that was first described in  \cite[Prop. 22]{serre:OpenImageThm}. 

Assume that $\Delta_E$ is not a square in $\QQ^\times$. Then,  $\QQ(\sqrt{\Delta_E})$ is the unique quadratic extension of $\QQ$ inside $\QQ(E[2])$.  Now, since $\QQ(\sqrt{\Delta_E})/\QQ$ is an abelian extension we know that there exists an integer $n\geq 3$ in such that $\QQ(\sqrt{\Delta_E}) \subseteq \QQ(\zeta_n)$ by the Kronecker--Weber theorem (e.g., $n=4|\Delta_E|$ works). 
The Weil pairing implies that $\QQ(\zeta_n)\subseteq\QQ(E[n])$, and thus $\QQ(E[2]) \cap \QQ(\zeta_n) \simeq \QQ(\sqrt{\Delta_E})$.  It follows that $\QQ(E[2])\cap \QQ(\zeta_n)$ is a non-trivial quadratic extension of $\QQ$, and therefore $\Gal(\QQ(E[2],\zeta_n)/\QQ)$ is a subgroup of index $2$ of $\Gal(\QQ(E[2])/\QQ)\times \Gal(\QQ(\zeta_n)/\QQ).$ This coincidence causes an entanglement between $\Q(E[2])$ and $\Q(E[n])$ which we call a Serre entanglement. In the notation of Definition \ref{defn:weil-type}, we have $K=\QQ(E[2])\cap \QQ(E[n])$ and $K\cap \QQ^\text{ab}=K\cap \Q(\zeta_n)=\Q(\sqrt{\Delta_E})$ which is a non-trivial quadratic extension of $\Q(\mu_2)=\Q$. Therefore, $E/\Q$ has a Weil $(2,4|\Delta_E|)$-entanglement of type $\Z/2\Z$, which we call a \cdef{Serre entanglement}. See Example \ref{ex-discriminant-weil-ent} for an instance of this phenomenon. 

\begin{remark}
	A Serre entanglement is a discriminant entanglement as in Definition \ref{defn:disc_entang}, part (2). Indeed, let $n\geq 3$ be as above such that it is minimal with the property $\Q(\sqrt{\Delta_E})\subseteq \Q(\zeta_n)$ and let $m=\lcm(2,n)$. For an integer $k\geq 2$, let $G_k = \Im(\rho_{E,k})$, and write $m=2^r\cdot s$ where $r,s\geq 1$ and $s$ is odd. Let $\mathcal{G}$ be the subgroup of $\GL(2,\Z/m\Z)\cong \GL(2,\Z/2^r\Z)\times \GL(2,\Z/s\Z)$ that isomorphic to $G_{2^r}\times G_s$. Then, we say there is a Serre entanglement if $G_m$ is a subgroup of $\mathcal{G}$ of index $2$, and there are index-$2$ normal subgroups $N_2$ and $N_n$ of $\mathcal{G}$ such that
	\begin{enumerate}
		\item $N_2$ corresponds to $\Q(\sqrt{\Delta_E})\subseteq \Q(E[2])$,
		\item $N_n$ corresponds to a quadratic subfield of $\Q(\zeta_n)\subseteq\Q(E[n])$, with $N_n\neq N_2$, and
		\item $N_2\cap G_m = N_n \cap G_m$.
	\end{enumerate}
\end{remark}
\begin{example} For instance, let $E/\Q$ be the elliptic curve with LMFDB label \texttt{21.a1}, and discriminant $\Delta_{E} = 3\cdot 7^2$. Hence, $\Q(\sqrt{\Delta_E})=\Q(\sqrt{3})\subseteq \Q(\zeta_3,i)=\Q(\zeta_{12})$, and therefore there is a $(2,12)$-Serre entanglement. In this case, $n=m=12$, and  $\mathcal{G}\subseteq \GL(2,\Z/12\Z)$ is isomorphic to $G_4\times G_3$, where $G_k = \Im(\rho_{E,k})$ as before. Finally, $N_2$, $N_n$, and $G_m$ are generated as follows, as subgroups of $\GL(2,\Z/12\Z)$: 
	$$N_2 =\left\langle 
		\begin{pmatrix}
	9&2\\2&1
	\end{pmatrix},
	\begin{pmatrix}
1&8\\4&7
\end{pmatrix},
	\begin{pmatrix}
1&4\\2&1
\end{pmatrix},
	\begin{pmatrix}
3&2\\8&7
\end{pmatrix},
	\begin{pmatrix}
5&2\\4&5
\end{pmatrix}
	\right\rangle,$$
	
$$N_n = \left\langle
	\begin{pmatrix}
5&5\\10&3
\end{pmatrix},
	\begin{pmatrix}
5&8\\2&3
\end{pmatrix},
	\begin{pmatrix}
11&7\\8&3
\end{pmatrix},
	\begin{pmatrix}
3&1\\8&7
\end{pmatrix}
\right\rangle,$$

$$G_m = \left\langle
	\begin{pmatrix}
7&9\\2&7
\end{pmatrix},
	\begin{pmatrix}
5&7\\4&9
\end{pmatrix},
	\begin{pmatrix}
3&1\\10&1
\end{pmatrix},
	\begin{pmatrix}
5&11\\10&9
\end{pmatrix}
\right\rangle.$$
One can verify that $N_2$ and $N_n$ are index-$2$ subgroups of $\mathcal{G}$, that $N_2\neq N_n$, and $N_2\cap G_m = N_n\cap G_m$. It is worth pointing out that $N_n$ fixes $\Q(\sqrt{3})\subseteq \Q(\zeta_{12})$ because $\det(N_n)=\{1,11\bmod 12\}$, and on the other hand, $\zeta_{12}+\zeta_{12}^{11} = \sqrt{3}$. Moreover, $\zeta_{12}+\zeta_{12}^{11}$ is invariant under the automorphism that sends $\zeta_{12}\mapsto \zeta_{12}^{11}$, and therefore $\Q(\sqrt{3})$ must be fixed by $N_{n}$, as desired.
\end{example}

If $E/\Q$ has CM, and $j\neq 1728$, then $\Delta_E$ is not a perfect square (see \cite[Appendix A, \S 3]{silverman2}; the discriminant changes by a square after a quadratic or cubic twist) and the same argument as in the first paragraph in this section shows that $E/\Q$ has a  Serre entanglement. If $j_E=1728$, then $E$ is given by a model $y^2=x^3-dx$, where $d$ is a non-zero fourth-power-free integer, and $\Q(E[2])=\Q(\sqrt{d})$. If $d$ is not a square (in other words, if $E$ is not a quadratic twist of $y^2=x^3-x$), then we once again have a Serre entanglement. If $d$ is a square, however, then $\Q(E[2])=\Q$ and there is no Serre entanglement involving the $2$-torsion. However, the $2$-adic image is ``small'' even for a CM curve i.e., there is a vertical collapsing, which we will discuss below (see Example \ref{exam:CMexample}). 

To summarize, we have shown the following result.

\begin{theorem}
	Let $E/\Q$ be an elliptic curve. 
		\begin{enumerate}
			\item If $\Delta_{E}$ is not a perfect square in $\Q^\times$, then $E/\Q$ has a Serre entanglement. 
			\item If $\Delta_E\in (\QQ^\times)^2$, then $E$ has a vertical collapsing.
		\end{enumerate} 
\end{theorem}

\begin{remark}\label{rem-deltaclass} If $d\in\QQ^\times$, then the elliptic curves over $\mQ$ with $\Delta_E \equiv d \mod (\QQ^\times)^2$ are precisely those curves $E/\QQ$ with $j$-invariant of the form $j(E) = dt^2+1728$ for some $t \in \QQ^\times$ (or $j(E)=1728$ if $d\equiv 3 \bmod (\QQ^\times)^2 $). Indeed, the discriminant of the universal elliptic curve
	$$E_{j_0}/\QQ : y^2+xy = x^3-\frac{36}{j_0-1728}x -\frac{1}{j_0-1728}$$
	with $j(E_{j_0})=j_0$ is given by $\Delta_{E_{j_0}} = \frac{j_0^2}{(j_0-1728)^3}$. Hence, $\Delta_{E_{j_0}}\equiv j_0-1728 \bmod (\QQ^\times)^2$ as long as $j_0\neq 0$ or $1728$. Moreover, if $E'/\QQ$ is any other elliptic curve with $j(E)=j_0$, then $\Delta_{E_{j_0}}$ and $\Delta_{E'}$ differ by a square in $\QQ^\times$. 
\end{remark} 
\begin{prop}\label{prop:infinte_Serre}
For a fixed quadratic number field $K$, there are infinitely many $\overline{\QQ}$-isomorphism classes of elliptic curves $E/\Q$ with a Serre entanglement such that $\Q(\sqrt{\Delta_E})=K$.
\end{prop}
\begin{proof}
	Let $K$ be a fixed quadratic field. Then, there is a square-free integer $d$ such that $K=\Q(\sqrt{d})$. It follows from Remark \ref{rem-deltaclass} that any elliptic curve $E/\Q$ with $j(t)=dt^2+1728$, for any $t\in\Q^\times$, has a Serre entanglement with $\Q(\sqrt{\Delta_E})=\Q(\sqrt{d})=K$, as desired.
\end{proof}

\begin{example}\label{ex-both-abelian-and-weil}
	Let $p>2$ be a prime, and let $\hat{p}=(-1)^{\frac{p-1}{2}}\cdot p$. For any $t\in \QQ^\times$, let $$\mathcal{E}_{p,t}\colon y^2+\hat{p}txy = x^3-36\hat{p}^3t^2x -\hat{p}^5t^4.$$ Then, $j(\mathcal{E}_{p,t})=\hat{p}t^2+1728$ and therefore (by our comments in Remark \ref{rem-deltaclass}) we have $\Delta_{\mathcal{E}_{p,t}}\equiv (-1)^{(p-1)/2}p \bmod (\QQ^\times)^2$. In particular, $\Q(\sqrt{\hat{p}})$ is the unique quadratic subfield of $\Q(\mathcal{E}_{p,t}[2])$. Since $\Q(\sqrt{\hat{p}})\subseteq \Q(\zeta_p)$, it follows that $\Q(\sqrt{\hat{p}})\subseteq \Q(\mathcal{E}_{p,t}[2])\cap \Q(\mathcal{E}_{p,t}[p])$ and an entanglement of Serre type ensues. Moreover, any elliptic curve $E'/\QQ$ with a Serre entanglement between $\Q(E'[2])$ and $\Q(E'[p])$ is a quadratic twist of a curve in the family $\mathcal{E}_{p,t}$. We also remark here that the entanglements described in this family are both of abelian and Weil-type, since $d=\gcd(2,p)=1$.
	\end{example} 

\subsubsection{{\bf Other discriminant $(2,q)$-entanglements}}\label{sec:other-disc-ent} Let $q$ be a prime such that $E/\Q$ has a $(2,q)$-discriminant entanglement, that is, $\Q(\sqrt{\Delta_E})\subseteq \Q(E[2])\cap \Q(E[q])$. Since $\Q(\sqrt{\Delta_E})$ (when non-trivial) is the unique quadratic subfield of $\Q(E[2])$, it follows that any abelian $(2,q)$-entanglement of type $\Z/2\Z$ is a discriminant entanglement. In \cite{danielsMorrow:Groupentanglements}, Daniels and Morrow have classified all the $(p,q)$-entanglements of type $S$, for some finite group $S$, such that there are infinitely non-$\overline{\Q}$-isomorphic elliptic curves over $\Q$ with such an entanglement. In particular, when $p=2$, they show that the  possible values for $q$ are $q\in \{3,5,7,13\}$. Thus, it remains to study discriminant $(2,q)$-entanglements with $q\not\in \{3,5,7,13\}$, that are not of Serre type.

\begin{lemma}\label{lem-borel_disc_entanglement}
	Suppose $E/\Q$ has no CM, and it has a discriminant $(2,q)$-entanglement where $\rho_{E,q}$ is contained in a Borel subgroup, for $q\not\in \{2,3,5,7,13\}$. Then, $q=11$, and $j=-11^2$ or $-11\cdot 131^3$. 
\end{lemma} 
\begin{proof} 
Suppose $E/\Q$ has no CM and an isogeny of degree $q\not\in \{2,3,5,7,13\}$, and that $E$ has a discriminant entanglement. By the classification of rational isogenies, we have $q\in \{11,17,37\}$ and the $j$-invariant of $E$ is one of finitely many, as given in \cite{lozano2013field}. Since quadratic twisting preserves the field of $x$-coordinates (Lemma 9.6 of \cite{lozano2013field}), if the quadratic fields contained in $\Q(E[q])$ are contained in $\Q(x(E[q]))$, then either all twists of $E$ have a discriminant entanglement or none have it. On the contrary, if the kernel of the $q$-isogeny is generated by $Q\in E[q]$ and $\Q(x(Q))$ is of odd degree, then there is a quadratic twist of $E$ with a discriminant entanglement. By Theorem 9.4 of \cite{lozano2013field}, for $p=17$ and $37$, the degree of $\Q(x(Q))$ is even, and therefore all the quadratic subfields  of $\Q(E[q])$ are contained in $\Q(x(E[q]))$, and it suffices to check one twist for each $j$-invariant. One can check that \texttt{14450.b1} and \texttt{14450.b2}, and \texttt{1225.b1} and \texttt{1225.b2} do not have discriminant entanglements. Therefore there are no discriminant entanglements of type $(2,17)$ and $(2,37)$.

However, for $p=11$, the degree of $\Q(x(Q))$, where $Q$ generates the kernel of an $11$-isogeny, is odd (degree $5$) and therefore there is a twist of $E$ with a discriminant entanglement. For instance, the elliptic curves \texttt{1936.a1} and \texttt{1936.b1} have discriminant entanglements.
\end{proof}

 We will come back to the other cases according to the image of $\rho_{E,q}$ when we prove Theorem \ref{thmx:main0} in Section \ref{sec-other}.

\subsection{\bf CM entanglements} \label{sec-CM-entanglements}
Next, we describe entanglements in the CM case. We note here that by \cite[Lemma 3.15]{BCSTorPointsOnCM}, if $E/\mQ$ is an elliptic curve with complex multiplication by an order $\mathcal{O}_{f,K}$ of an imaginary quadratic field $K$, and $n\geq 3$, then $K\subseteq \Q(E[n])$. Hence, for any $a,b\geq 3$, we at least have $K\subseteq \mQ(E[a])\cap \mQ(E[b])$ which results in a Weil entanglement of type $T$ with $\Z/2\Z\subseteq T$.

Therefore, the presence of CM ensures that we have an entanglement, and we record this phenomena below. 

\begin{defn}\label{defn:CM-entanglement}
Let $E/\mQ$ be an elliptic curve with complex multiplication by the order $\mathcal{O}_{f,K}$ with $K/\mQ$ an imaginary quadratic extension. 
We say that an elliptic curve $E/\mQ$ has a \cdef{CM  $(a,b)$-entanglement of type $\zZ{2}$} if $K \subseteq \mQ(E[a])\cap \mQ(E[b])$. 
\end{defn}

In addition to the CM entanglement, elliptic curves with CM may have other vertical collapsing and horizontal entanglement that we shall describe below. First, we need to describe the image of the Galois representations attached to elliptic curves with CM.

\begin{defn}\label{defn-CMimage} 
Let $K$ be an imaginary quadratic field with discriminant $\Delta_K$ and let $\OO_{K,f}$ be the order of $K$ of conductor $f\geq 1$.
	Let $E/\Q$ be an elliptic curve with CM by $\OO_{K,f}$, let $n\geq 3$, and let $\rho_{E,n}$ be the mod $n$ image of Galois.  
	
We define groups of $\GL(2,\Z/n\Z)$ as follows:
	\begin{itemize}
		\item If $\Delta_Kf^2\equiv 0\bmod 4$,  let $\delta=\Delta_K f^2/4$, and $\phi=0$.
		\item If $\Delta_Kf^2\equiv 1 \bmod 4$, let $\delta=\frac{(\Delta_K-1)}{4}f^2$, let $\phi=f$.
	\end{itemize}
The \cdef{Cartan subgroup} $\cC_{\delta,\phi}(n)$ of $\GL(2,\Z/n\Z)$ is defined by
	$$\cC_{\delta,\phi}(n)=\left\{\left(\begin{array}{cc}
	a+b\phi & b\\
	\delta b & a\\
	\end{array}\right): a,b\in\Z/n\Z,\  a^2+ab\phi-\delta b^2 \in (\Z/n\Z)^\times \right\},$$
	and we also define
	\[
	\mathcal{N}_{\delta,\phi}(n) = \left\langle \cC_{\delta,\phi}(n),\left(\begin{array}{cc} -1 & 0\\ \phi & 1\\\end{array}\right)\right\rangle.
	\] Finally, let
	$$\mathcal{N}_{\delta,\phi}(p^\infty)=\varprojlim \mathcal{N}_{\delta,\phi}(p^n) \quad  \text{ and } \quad \mathcal{N}_{\delta,\phi}(\widehat{\Z})=\varprojlim \mathcal{N}_{\delta,\phi}(n).$$
\end{defn} 

Now we are ready to define a vertical CM tanglements and horizontal CM entanglements in terms of the groups of Definition \ref{defn-CMimage}.
\begin{defn}\label{defn-CMverticalentanglement} 	
	Let $E/\Q$ be an elliptic curve with CM by an order of an imaginary quadratic field $K$, and let $\mathcal{N}_{\delta,\phi}(\cdot)$ be the groups of Definition \ref{defn-CMimage}.
	\begin{enumerate}
		\item We say that $E/\Q$ has a \cdef{vertical CM collapsing} if $\Im (\rho_{E,p^\infty})$ is contained but not equal to a conjugate subgroup of $\mathcal{N}_{\delta,\phi}(p^\infty)$ in $\GL(2,\Z_p)$.
		\item Let $n\geq 6$ be an integer divisible by at least two distinct primes. We say that $E/\Q$ has a \cdef{horizontal CM entanglement} if $\Im (\rho_{E,p^\infty})$ is contained in but not equal to a conjugate subgroup of $\mathcal{N}_{\delta,\phi}(n)$ in $\GL(2,\Z/n\Z)$, and if the entanglement is not explained by a product of vertical tanglements.
	\end{enumerate} 
\end{defn}

\begin{remark}
	By \cite[Theorem 1.1]{lozanoCMreps}, there is always a $\Z/n\Z$-basis of $E[n]$ such that the image of $\rho_{E,n}$ is contained in $\mathcal{N}_{\delta,\phi}(n)\subseteq \GL(2,\Z/n\Z)$. Moreover, the index of $\Im( \rho_{E,n})$ in $\mathcal{N}_{\delta,\phi}(n)$ is a divisor of $4$ or $6$, and if $j\neq 0,1728$, then the index divides $2$. In fact, the same  is true for the adelic image as the index is a divisor of $|\OO_{K,f}^\times|$. This has been shown in \cite[Theorem 1.5]{lombardo}, a stronger version in \cite[Corollary 1.5]{bourdon-clark}, and in a slightly different way in \cite[Theorem 1.2]{lozanoCMreps}. 
\end{remark}

\begin{example}\label{exam:CMexample}
According to \cite[Theorem 1.7]{lozanoCMreps}, there is a $\Z_2$-basis of $T_2(E)$ such that the $2$-adic image of an elliptic curve with $j_E=1728$ is contained in the subgroup
$$\mathcal{N}_{-1,0}(2^\infty)=\left\langle \left(\begin{array}{cc} 1 & 0\\ 0 & -1\\\end{array}\right), \left\{\left(\begin{array}{cc} a & b\\ -b & a\\\end{array}\right)\in \GL(2,\Z_2) : a^2+b^2\not\equiv 0 \bmod 2 \right\} \right\rangle$$
of $\GL(2,\Z_2)$. In particular, the image of $\rho_{E,2}$ is contained in the subgroup of order $2$ given by
$$\left\{ \left(\begin{array}{cc} 1 & 0\\ 0 & 1\\\end{array}\right), \left(\begin{array}{cc} 0 & 1\\ 1 & 0\\\end{array}\right)\right\}\subseteq \GL(2,\Z/2\Z).$$
Hence, if $\Q(E[2])=\Q$, then the image of $\rho_{E,2}$ is trivial, and therefore the image of $\rho_{E,2^\infty}$ is at least of index $2$ in $\mathcal{N}_{-1,0}(2^\infty)$. Hence, there is a vertical collapsing. 
\end{example}

\subsection{{\bf Fake CM entanglements}}\label{sec-fakecm} Here we describe the last of our generic or explained abelian entanglements, which we call a \cdef{fake CM entanglement}. Let $E/\Q$ be an elliptic curve {\it without CM} and let $p>2$ be an odd prime such that the image of $\rho_{E,p}$ is conjugate to $\mathcal{N}(p)$, the normalizer of a (split or non-split) Cartan subgroup $\mathcal{C}(p)$ of $\GL(2,\FF_p)$. Then, there is an element $\tau \in \Gal(\Q(E[p])/\Q)$ of order $2$, such that $\mathcal{N}(p) = \mathcal{C}(p) \rtimes \langle \tau \rangle$. In particular, the Cartan subgroup is normal in $\mathcal{N}(p)$, and the fixed field  by $\mathcal{C}(p)$ is a quadratic extension $\widehat{K}_p=\Q(E[p])^{\mathcal{C}(p)}$ of $\Q$. The Weil pairing gives another quadratic subfield in $\Q(E[p])$, namely $F_p=\Q(\sqrt{\hat{p}})$, where $\hat{p} = (-1)^{(p-1)/2}\cdot p$. Moreover, $\widehat{K}_p\neq F_p$. Indeed, let $H$ be the subgroup of index $2$ of $\mathcal{N}(p)$ that fixes $F_p$. Since $F_p\subseteq \Q(\zeta_p)$, it follows that $\mathcal{N}(p)\cap \SL(2,\FF_p)\subseteq H$. However, $\mathcal{C}(p)$ does not contain $\mathcal{N}(p)\cap \SL(2,\FF_p)$. Indeed, $\mathcal{C}(p)$ contains at least one element $c$ of determinant $-1$, and $c\tau\in  \mathcal{N}(p)\cap \SL(2,\FF_p)$ but not in $\mathcal{C}(p)$. Further, we note that $-\operatorname{Id}$ is contained in $\mathcal{C}(p)$, and therefore $\widehat{K}_p\subseteq \Q(x(E[p]))$. In particular, $\widehat{K}_p$ is contained in $\Q(E'[p])$ for any quadratic twist $E'$ of $E/\Q$ (see, for instance, \cite[Lemma 9.6]{lozano2013field}). In particular, there is a square-free integer $d$, different from $\hat{p}$, such that $\widehat{K}_p=\Q(\sqrt{d})$. Further, there is a third quadratic subfield of $\Q(E[p])$, namely $\Q(\sqrt{d\cdot \hat{p}})$. 

More generally, let $E/\Q$ be an elliptic curve without CM such that the image of $\rho_{E,p}$ is contained in the normalizer of a Cartan subgroup of $\GL(2,\FF_p)$, and let $K_p(E)=\Q(E[p])\cap \Q^{\text{ab}}$. Then, Theorem 5.12 and Corollary 5.17 of \cite{danielsLR:coincidences} show that $\Gal(K_p(E))\cong (\Z/p\Z)^\times$ or $\Z/2\Z\times (\Z/p\Z)^\times$. In the first case, the only quadratic subfield of $\Q(E[p])$ is $\Q(\sqrt{\hat{p}})$. In the second case, when $\Gal(K_p(E))\cong \Z/2\Z\times (\Z/p\Z)^\times$, there are three quadratic subfields in $\Q(E[p])$, namely $\Q(\sqrt{\hat{p}})$, $\widehat{K}_p = \Q(\sqrt{d})$, and  $\Q(\sqrt{d\cdot \hat{p}})$. 

\begin{defn}\label{defn:fakeCM}
	Let $E/\Q$ be an elliptic curve without CM such that the image of $\rho_{E,p}$ is contained in the normalizer of a Cartan subgroup of $\GL(2,\FF_p)$, and suppose that $\Gal(K_p(E)) \cong \Z/2\Z\times (\Z/p\Z)^\times$. We say that $E/\Q$ has a \cdef{fake CM entanglement} if it has an abelian $(p,q)$-entanglement of type $\Z/2\Z$ where $\Q(E[p]) \cap \Q(E[q])\cong \widehat{K}_p$ or $\Q(\sqrt{d\cdot \hat{p}})$ as defined above. 
\end{defn}

\begin{example}
	The elliptic curve with LMFDB label \href{https://www.lmfdb.org/EllipticCurve/Q/338/d/2}{\texttt{338.d2}} and $j(E)=11^3/2^3$ is an example of a curve with a fake CM $(3,5)$-entanglement of type $\Z/2\Z$, where the entanglement field is $\widehat{K}_3 = \Q(\sqrt{13})$. In this case, the elliptic curve has normalizer of split Cartan image at $p=3$ and it is Borel at $q=5$. Since $j(E)$ is not integral, the elliptic curve is not a CM elliptic curve.
\end{example}

\section{Proofs of Theorem \ref{thmx:main0} and Corollary \ref{cor-main0}} \label{sec-other}
\label{sec-proofs_main_theorems}

Before we prove Theorem \ref{thmx:main0}, we introduce some work of Lemos and by two of the authors which we will use in the proof. Recall that Serre asked, for an elliptic curve over $\Q$ without CM, whether $\rho_{E,p}$ is surjective for all primes $p>37$.

\begin{theorem}
	[\cite{lemos}]\label{thm-lemos} Let $E/\Q$ be an elliptic curve over $\Q$ without CM, and let $p$ be a prime. 
	\begin{enumerate}
		\item Suppose that there is a prime $q$ such that the image of $\rho_{E,q}$ is contained in either a Borel subgroup or the normalizer of a split Cartan subgroup. Then, $\rho_{E,p}$ is surjective for $p>37$.
		\item Suppose that there is a prime $q$ such that the image of $\rho_{E,q}$ is contained in the normalizer of a split Cartan subgroup and a prime $p\geq 11$ such that the image of $\rho_{E,q}$ is contained in the normalizer of a non-split Cartan subgroup. Then, $j(E)$ is integral.
	\end{enumerate}
\end{theorem}

 Suppose $E/\Q$ is an elliptic curve, and $p,q$ are distinct primes such that $E$ has an abelian $(p,q)$-entanglement of type $S$. Then, it follows that $F=(\Q(E[p])\cap \Q(E[q]))\cap \Q^{\text{ab}}$ is a Galois extension of $\Q$ with $\Gal(F/\Q)\cong S$. Hence, $F\subseteq \Q(E[\ell])\cap \Q^{\text{ab}}$ for $\ell=p,q$. The first and second author have classified the possibilities for $\Q(E[\ell])\cap \Q^{\text{ab}}$ in \cite{danielsLR:coincidences}, Theorem 1.6 and Corollaries 5.3, 5.7, 5.11, 5.17, and 5.20. 

\begin{theorem}[\cite{danielsLR:coincidences}]\label{thm-coincidences}
	Let $E/\QQ$ be an elliptic curve and let $p$ be an odd prime. Let $K_p(E)=\Q(E[p])\cap \Q^{\ab}$. Then:
	\begin{enumerate}
		\item $\Gal(K_p(E)/\Q)\simeq (\Z/p\Z)^\times \times C$, where $C$ is a cyclic group of order dividing $p-1$. 
		\item If $E/\Q$ does not have a rational $p$-isogeny (equivalently, the image of $\rho_{E,p}$ is not contained in a Borel subgroup), then $C$ is trivial or of order 2 and $K_p(E)=F(\zeta_p)$ with $F/\Q$ a trivial or quadratic extension. 
		\item If the image of $\rho_{E,p}$ is exceptional or full, then $K_p(E)=\Q(\zeta_p)$.
	\end{enumerate} 
Finally, if $p=2$, then $K_p(E)$ is a trivial, quadratic, or cubic extension of $\Q$.
\end{theorem}

\subsection{Proof of Theorem \ref{thmx:main0}}

	Suppose $E/\Q$ is an elliptic curve without CM and let $p<q$ be primes such that $E$ has an abelian $(p,q)$-entanglement. Let $F=\Q(E[p])\cap \Q(E[q])\cap \Q^{\text{ab}}$. Then, $F\subseteq K_p(E)\cap K_q(E)$, where $K_s(E)=\Q(E[s])\cap \Q^{\text{ab}}$. We distinguish several cases:
	\begin{enumerate}
		\item If $p=2$, then $F\subseteq K_2(E)\subseteq \Q(E[2])$ can be quadratic or cyclic cubic. If $F$ is quadratic, then there is a discriminant entanglement (see Sections \ref{subsec:disc_entangles} and \ref{sec:other-disc-ent}). 
		
		\item If $p=2$, and $F$ is a cyclic cubic, then  $d_F=[F:\Q]=3$. In this case $F=\Q(E[2])$ is a cyclic cubic number field, and $F$ appears also in $\Q(E[q])$. By Theorem \ref{thm-coincidences}, the curve $E$ must have a rational $q$-isogeny where $3$ divides $q-1$. 
		By the classification of cyclic rational isogenies (see, for instance, Section 9 of \cite{lozano2013field} and in particular Theorem 9.5), we must have $q\in \{7,13,19,37,43,67,163\}$. Since $X_0(\ell)$ with $\ell\in \{19,37,43,67,163\}$ only has finitely many rational points, these choices of $q$ lead to, at most, finitely many examples ($j$-invariants). If $\ell=43,67,163$ the only $j$-invariants are CM ones, but we are assuming $E$ does not have CM. If $\ell=19$ or $37$, then there are three possible $j$-invariants, and one can easily verify that elliptic curves with those three $j$-invariants have full $2$-adic image, and therefore $\Q(E[2])$ is not cyclic cubic.
		
		If $q=13$, we have used \texttt{Magma} to look for subgroups of $\GL(2,\Z/26\Z)$ that would surject to the corresponding images mod $2$ and mod $13$, with a $\Z/3\Z$-entanglement. Our computations show that there are precisely $4$ such subgroups, and none of these define modular curves of genus $0$ or $1$ (they have genus $3,3,3$, and $5$, respectively), and any elliptic curve with this prescribed entanglement must correspond to a $\mQ$-rational point on a modular curve of genus $3$ or $5$. Therefore, Faltings' theorem tells us there are at most finitely many $j$-invariants of elliptic curves with a $(2,13)$-entanglement of type $\Z/3\Z$. 
		
		If $q=7$, however, there are such entanglements, and in fact there are such entanglements with the added condition that $F$ is not a subfield of $\Q(\zeta_7)$, which makes the entanglement not of Weil type. These entanglements have been described in \cite{danielsMorrow:Groupentanglements}. 
		The authors show that there are three genus $0$ modular curves whose $\Q$-points correspond to elliptic curves with such an abelian entanglement, and their parametrizations appear in Theorem \ref{thmx:main0}.

		\item If $2<p<q$, and the images of $\rho_{E,p}$ and $\rho_{E,q}$ are exceptional or full, then Theorem \ref{thm-coincidences} implies that $F\subseteq K_p(E)\cap K_q(E)=\Q(\zeta_p)\cap \Q(\zeta_q) = \Q$, and therefore the entanglement would be trivial.
		
		\item If both $p$ and $q$ are odd, and the images of $\rho_{E,p}$ and $\rho_{E,q}$ are contained in Borel subgroups, then $E$ has a rational cyclic $pq$-isogeny.  By the classification of cyclic rational isogenies (see, for instance, Section 9 of \cite{lozano2013field} and in particular Theorem 9.5), we must have $pq=15$, and the only $j$-invariants with a $15$-isogeny are $j\in \{-5^2/2, -5^2\cdot 241^3/2^3,-5\cdot 29^3/2^5,5\cdot 211^3/2^{15} \}.$  In all four cases, the $3$-torsion is defined over a field of degree dividing $6$ while the $5$-torsion is defined over a field of degree dividing $20$ so the intersection must be trivial or quadratic. As we show in Example \ref{exam:nonSerreWeilCM}, quadratic entanglements are possible in this case.
		
		\item The last case remaining is $p,q$ odd, with $d_F=2$, but without rational $pq$-isogenies. Thus, $F=(\Q(E[p])\cap \Q(E[q]))\cap \Q^{\text{ab}}$ is a quadratic extension of $\Q$, and since there is no $pq$-isogeny either the mod-$p$ or mod-$q$ image is not contained in a Borel subgroup. Thus, one or both images are contained in a normalizer of a Cartan subgroup, and therefore, the entanglement is either of Weil type or of fake CM type.

	\end{enumerate} 
This completes the proof of Theorem \ref{thmx:main0}. 

	\begin{example}  
	For instance, the curve with LMFDB label \href{https://www.lmfdb.org/EllipticCurve/Q/1922/c/2}{\texttt{1922.c2}} has an abelian $(2,7)$-entanglement of type $\Z/3\Z$, which is not Serre, Weil, or CM. 
	Indeed, the entanglement field $F=\Q(E[2])$ is the cyclic cubic field that corresponds to $x^3-x^2-10x+8=0$, with discriminant $31^2$, while the field of definition of one kernel of the $7$-isogeny is a cyclic sextic field with discriminant $-31^5$ that contains $F$ as its unique cubic subfield.
\end{example} 

\begin{example}\label{exam:nonSerreWeilCM}
	Let $d\neq -2,-3, 5$ be a square-free integer and let $E^{(d)}/\Q$ be the quadratic twist by $d$ of the curve $E\colon y^2+xy+y = x^3-126x-552$ with LMFDB label \href{https://www.lmfdb.org/EllipticCurve/Q/50/a/1}{\texttt{50.a1}} with $j(E)=-5^2\cdot 241^3/2^3$. The image of the representation $\rho_{E,p}$ for $p=3,5$ is of the form
	$$\left\{\left(\begin{array}{cc}
	\chi_p & b\\
	0 & 1\\
	\end{array}\right) : b \in \ZZ/p\ZZ \right\}\subseteq \GL(2,\FF_p).$$
	Hence, if $\chi_d$ is the quadratic character associated to $\Q(\sqrt{d})$, then the image of $\rho_{E^{(d)},p}$ is of the form
	$$\left\{\left(\begin{array}{cc}
	\chi_p\chi_d & b\\
	0 & \chi_d\\
	\end{array}\right) : b \in \ZZ/p\ZZ \right\}\subseteq \GL(2,\FF_p).$$
	In particular, $\QQ(E^{(d)}[p])\cap \QQ^\text{ab} = \QQ(\zeta_p,\sqrt{d})$ for $p=3,5$, and so $\QQ(E^{(d)}[3])\cap \QQ(E^{(d)}[5])=\QQ(\sqrt{d})$. 
	
	Note that:
	\begin{itemize}
		\item  $E$ is a non-CM curve, so its twists $E^{(d)}$ are non-CM as well. Thus the entanglement is not of CM type. 
		\item The entanglement is between the $3$rd and $5$th division fields, so it is not a discriminant entanglement (since it does not involve the $2$nd division field). Moreover, since $\Delta_E = -1\cdot 2^3\cdot 5^4$, we have that $\QQ(\sqrt{\Delta_{E^{(d)}}})=\QQ(\sqrt{-2})$, and hence there is a discriminant entanglement between $2$ and $p=3,5$ if and only if $d\neq -2$.
		\item Finally, $\QQ(\sqrt{d})$ is not contained in $\QQ(\zeta_p)$ for $p=3,5$, as long as $d\neq -3,5$. Hence, the entanglement is not of Weil type. We notice, however, that if $n$ is the smallest positive integer such that $\QQ(\sqrt{d})\subseteq \QQ(\zeta_n)$, then there is an additional $(p,n)$-Weil entanglement of type $\ZZ/2\ZZ$.  
	\end{itemize}
\end{example}

\subsection{Proof of Corollary \ref{cor-main0}} In order to prove the Corollary, we need to analyze all the possibilities of quadratic entanglements in more depth.  

	\begin{enumerate}
	\item If $p=2$, then $F\subseteq K_2(E)\subseteq \Q(E[2])$ can be quadratic or cyclic cubic. If $F$ is quadratic, then there is a discriminant entanglement. If the image of $\rho_{E,q}$ is exceptional or full, then Theorem \ref{thm-coincidences} implies that $F\subseteq K_q(E)=\Q(\zeta_q)$ and therefore the entanglement is of Serre type. If the image of $\rho_{E,q}$ is contained in a Borel subgroup, then either $q\in \{3,5,7,13\}$ and the possibilities have been parametrized by \cite{danielsMorrow:Groupentanglements}, or Lemma \ref{lem-borel_disc_entanglement} shows that $\ell=11$ and there are two possible $j$-invariants. Otherwise, the image of $\rho_{E,q}$ is contained in the normalizer of a Cartan subgroup (but not in a Borel) and therefore the entanglement is either of Serre type (if $F=\Q(\sqrt{\hat{p}})$), or of fake CM type. If the entanglement is of fake CM type, and it corresponds to an infinite family, those once again appear in \cite{danielsMorrow:Groupentanglements} and those do occur for $q=3,5,7$. Otherwise, the corresponding modular curves are of genus $2$ or higher and contain at most finitely many rational points. Moreover, if we assume Serre's uniformity then $q\leq 37$, so there are only finitely many possibilities in total.
	
			\item The last case remaining is $p,q$ odd, with $d_F=2$, but without rational $pq$-isogenies, and $F=(\Q(E[p])\cap \Q(E[q]))\cap \Q^{\text{ab}}$ is a quadratic extension of $\Q$. First we study the case when one of the images is contained in a Borel subgroup $B$ and another image is contained in a normalizer of a Cartan subgroup. Since we are assuming that the elliptic curve does not have CM, then the classification of rational cyclic isogenies over $\Q$ together with Theorem \ref{thm-lemos}, show that $p,q\leq 37$ (without the need of Serre's uniformity).

	Note that if the image of $\rho_{E,q}$, say, is contained in a Borel, then there is either one or three non-trivial quadratic subfields of $\Q(E[q])$, where one of them is $\Q(\sqrt{\hat{q}})$ with $\hat{q} = (-1)^{(q-1)/2}\cdot q$. Hence, it may occur that $F=\Q(E[p])\cap \Q(E[q])\cap \Q^{\text{ab}}$ is one of the three quadratic subfields, and if $F\neq \Q(\sqrt{\hat{q}})$, then $E/\Q$ has an abelian $(p,q)$-entanglement of type $\Z/2\Z$ that is not a Weil entanglement.  Since $E/\Q$ does not have CM, then it is not a CM entanglement, and since $p,q$ are odd, it is not a discriminant entanglement either, thus not of Serre type. Since we are assuming the other image, that of $\rho_{E,p}$, is contained in the normalizer of a Cartan subgroup, it would be an entanglement of fake CM type. If there are infinitely many $j$-invariants that correspond to one of these entanglements, then there would be a modular curve of genus $0$ or $1$ (with positive rank) that corresponds to such mod $pq$ image. The tables in \cite{Daniels2019} show that these are the possibilities that  occur in genus $0$ and $1$:
	\begin{itemize}
		\item $(3\text{Nn},5\text{B})$ corresponds to a modular curve of genus $0$ which parametrize elliptic curves with $j$-invariant $j(t)=(3125t^6+250t^3+1)^3/t^3$. Using the results of \cite{danielsMorrow:Groupentanglements}, we know that there are exactly 2 groups representing entanglements between 3- and 5-division fields of this type with infinitely many points. They are both genus $0$ curves that parametrize (3,5)-entanglements of type $\ZZ/2\ZZ$, and their parametrizations appear in Corollary \ref{cor-main0}. 
		\item $(3\text{Ns},5\text{B})$ corresponds to a modular curve of genus $1$ but the rank is $0$. Thus, there are only two such $j$-invariants, which were computed in \cite{Daniels2019}, namely $11^3/2^3$ and $-29^3\cdot 41^3/2^{15}$. The elliptic curves with LMFDB labels \href{https://www.lmfdb.org/EllipticCurve/Q/338/d/2}{\texttt{338.d2}} and \href{https://www.lmfdb.org/EllipticCurve/Q/338/b/1}{\texttt{338.b1}}, with $j$-invariants  $j(E)=11^3/2^3$ and $-29^3\cdot 41^3/2^{15}$ respectively, are examples of curves with a $(3,5)$-entanglement of type $\Z/2\Z$, where the entanglement field is $F = \Q(\sqrt{13})$.
	\end{itemize}
	Any other possible combination of $p$ and $q$ would correspond to a modular curve of genus $\geq 2$, and since $p,q\leq 37$ as explained before, it follows that there are only finitely many $j$-invariants that could possibly have a fake CM entanglement other than those in the infinite families described above.
	
	\item The last case to study is when both images are contained in normalizer of Cartan subgroups. Hence, the entanglement is either of Weil type or of fake CM type. By work of Bilu, Parent, and Rebolledo, if the image of $\rho_{E,p}$ is contained in a normalizer of a split Cartan subgroup, then $p\leq 11$. So either $q>11$ and the image of $\rho_{E,q}$ is contained in a normalizer of a non-split Cartan, or $p,q\leq 11$. 
	
	In \cite{Daniels2019}, Daniels and Gonz\'alez-Jim\'enez have classified all the possible combinations of mod $p$ and mod $q$ images that are associated to modular curves (of level $pq$) of genus $0$ or $1$ (with positive rank). There they find that the only possibilities of mod $p$ and mod $q$ images which occur for infinitely many $j$-invariants are $(3\text{Nn},5\text{Nn})$, $(3\text{Nn},5\text{Ns})$, and $(3\text{Nn},7\text{Nn})$ which are parametrized by elliptic curves of positive rank (see Tables 9 and 10 of \cite{Daniels2019} for the equations). However, if a non-Weil entanglement did occur, then there is a quadratic entanglement $F=\Q(E[p])\cap \Q(E[q])\cap \Q^{\text{ab}}$. We have computed the genus of the quotients of images $(3\text{Nn},5\text{Nn})$, $(3\text{Nn},5\text{Ns})$, and $(3\text{Nn},7\text{Nn})$ where such a coincidence would occur, and none have genus $<2$ (in fact, no subgroup of these images has genus $<2$). Therefore, if any such examples exist, they do only for finitely many $j$-invariants. 
	
	It follows that if we assume Serre's uniformity question, then $p,q\leq 37$, and therefore there would be only finitely $j$-invariants that correspond to a fake CM entanglement such that the images of $\rho_{E,p}$ and $\rho_{E,q}$ are contained in a normalizer of a Cartan subgroup.
	\end{enumerate} 
This completes the proof of Corollary \ref{cor-main0}.

\section{\bf Entanglements of Weil-type}
\label{sec:explainedentanglements}
In this section, we prove several results on entanglements of Weil-type, and to conclude, we will present two questions which relates explained entanglements to generic polynomials. 

\subsection{Weil $(2,n)$-entanglements of $\ZZ/3\ZZ$-type}\label{subsec:explained_2,n_type3} 
This type of entanglement occurs in the case when $\QQ(E[2])/\QQ$ is an $\ZZ/3\ZZ$-extension. Elliptic curves with this kind of entanglement have already been studied in \cite[Prop. 8.4]{morrow2017composite}, which uses work of Gauss and Rubin--Silverberg. We note here that if $\QQ(\zeta_n)$ contains a cyclic cubic number field $L$, then there is an odd prime $p$ and $e\geq 1$ such that $L\subseteq \QQ(\zeta_{p^e})\subseteq \QQ(\zeta_n)$ and, in fact, either $p^e=9$ or $p^e=p$ if $p>3$ (note that, of course, not every cyclic cubic extension has prime conductor, but here for simplicity we concentrate on prime conductor cyclic cubics). In particular, it follows that if an elliptic curve $E/\QQ$ has a non-trivial $((2,n),\ZZ/3\ZZ)$-entanglement (in the notation of Definition \ref{defn:primitiveentanglement}), then $E$ also has a $((2,p^e),\ZZ/3\ZZ)$ entanglement, and $((2,p^e),\ZZ/3\ZZ)\leq ((2,n),\ZZ/3\ZZ)$. Thus, it suffices to realize $((2,p),\ZZ/3\ZZ)$ entanglements for $p>3$, and $((2,n),\ZZ/3\ZZ)$ for $n=9$.

\begin{proposition} \label{prop:explained_2,n_type3} 
Let $p$ be a prime, and let $e\geq 1$ such that $3\mid \varphi(p^e)$. Then, there exist infinitely many $\overline{\QQ}$-isomorphism classes of elliptic curves $E/\QQ$ (which are explicitly parametrized) such that $\QQ(E[2])\cap \QQ(E[p^e])\cong L$, where $L$ is the degree $3$ subfield of $\QQ(\zeta_{p^e})$.
\end{proposition}
\begin{proof}
	If $p>3$, then $p\geq 7$ and $3\mid p-1$. This case is taken care of in \cite[Prop. 8.4]{morrow2017composite}. Otherwise, suppose $p^e=9$. Then, $L\subseteq \QQ(\zeta_9)$ is the cubic field generated by the roots of the polynomial $f(x)=x^3 - 6x^2 + 9x - 3$. Let $E/\QQ$ be the elliptic curve given by $y^2=f(x)$ so that $\QQ(E[2])=L$. Following a procedure detailed in  \cite{rubinsilverberg} we can find a family of elliptic curves $\mathcal{E}_t/\QQ(t)$, one for each $t\in\QQ$, given by
	$$\mathcal{E}_t \colon y^2 =  x^3 - 3888(2303t^2 + 1)x - 46656(-2303t^3 - 6909t^2 + 3t + 1),$$
	such that $\QQ(\mathcal{E}_t[2])=\QQ(E[2])=L$ (further, the $j$-invariant of $E$ is non-constant as a function of $t$, so there are infinitely many non-isomorphic elliptic curves in the family). Since $L\subseteq \QQ(\zeta_9)\subseteq \QQ(\mathcal{E}_t[9])$, it follows that $\QQ(\mathcal{E}_t[2])\cap \QQ(\mathcal{E}_t[9])=L\cap \QQ(\mathcal{E}_t[9])=L\subseteq \Q(\zeta_9)$, as desired.
\end{proof}

\begin{remark}
Proposition \ref{prop:explained_2,n_type3} can be made explicit in the case where $p$ is prime such that $3\mid p-1$ using the \texttt{Magma} intrinsic \texttt{RubinSilverbergPolynomials}.
\end{remark}

\subsection{Weil $(3,n)$-entanglements of $\ZZ/2\ZZ$-type}\label{subsec:explained_3,n_type2}
We now turn our attention to $(3,n)$-entanglements of Weil-type.
We note that there are two possible ways to construct elliptic curves whose mod $3n$ representation has an image representing a Weil $(3,n)$-entanglement of $\ZZ/2\ZZ$-type.

First, we can construct elliptic curves such that either $\QQ(\zeta_3) \cap \QQ(E[n])$ is larger than expected (i.e., not equal to $\mQ$) or $\QQ(\zeta_n)\cap\QQ(E[3])$ is larger than expected. 
It turns out that the easiest way for one of these two things to be true is for $\Im(\rho_{E,n})$ (or resp.~$\Im(\rho_{E,3})$) to not be surjective; note that in the $n=2$ case, this produces a Serre $(2,3)$-entanglement which has been classified in Proposition \ref{prop:infinte_Serre}. 
Indeed, if $n\geq 3$ and $\rho_{E,n}$ is surjective, then $\Gal(\QQ(E[n])/\QQ)\simeq \GL(2,\ZZ/n\ZZ)$, and so we can determine the intersection
\[\QQ(E[n]) \cap \QQ^{\ab} =\begin{cases}
 \QQ(\zeta_n) & n \hbox{ is odd}\\	
  \QQ(\zeta_n,\sqrt{\Delta_E}) & n \hbox{ is even},\\
\end{cases}
\] 
where $\Delta_E$ is the discriminant of the elliptic curve in question.
Moreover, this tells us that it is more difficult to find anything larger than expected in those two intersections.
Assuming a positive answer to Serre's uniformity question (that is, $\Im(\rho_{E,p})$ is surjective onto $\GL(2,\Z/p\Z)$ for all primes $p>37$, for non-CM curves), there is an upper bound on how large $n$ can be if $\Im(\rho_{E,n})$ is not surjective. 
We analyze both cases separately. 

Suppose that $\Im(\rho_{E,3})$ is not surjective, and in particular that $\Im(\rho_{E,3})$ is conjugate to a subgroup of the full Borel subgroup of $\GL(2,\ZZ/3\ZZ)$ (i.e., $E$ has a rational $3$-isogeny).
Let $G=B(3)$ be the Borel subgroup of $\GL(2,\ZZ/3\ZZ)$.  We have that $\QQ(E[3])$ contains up to $3$ different quadratic extensions. 
If $P$ is a generator of the kernel of a 3-isogeny, then we can see that these quadratic extensions are $\QQ(\sqrt{-3})$, $\QQ(P)$, and the compositum of these two fields. 

We want to understand what extensions (either quadratic or trivial) arise as the field $\QQ(P)$. To begin, we start with a concrete example. 

\begin{example}\label{exam:explained_3,n}
Let $E/\QQ$ be the elliptic curve with LMFDB label \href{https://www.lmfdb.org/EllipticCurve/Q/19/a/1}{\texttt{19.a1}}. 
This elliptic curve has a rational $3$-isogeny and discriminant $\Delta_E = -19$. 
The kernel of the $3$-isogeny is generated by \[P = \left(-\frac{49}{3} , -\frac{9+\sqrt{-3}}{18} \right),\] and we have that $\QQ(x(P))=\QQ$ and $\QQ(P) = \QQ(\sqrt{-3})$. 
For a quadratic twist $E^{(d)}$, we denote the corresponding generator of the kernel of the 3-isogeny by $P_d$, so that $\QQ(P_d) = \QQ(\sqrt{-3d}).$ 
Thus, we have that for any square-free $d\in\Z$, there is an elliptic curve with a $3$-isogeny whose kernel is defined over $\QQ(\sqrt{d})$, namely,  one can just twist the elliptic curve $E$ by $-3d$. 
Therefore, we are able to realize any quadratic extension as the field of definition of the kernel of the $3$-isogeny through appropriate twisting. 
\end{example}

Next, we extend the technique from Example \ref{exam:explained_3,n}, to produce an infinite family of non-isomorphic elliptic curves, which have a Weil $(3,n)$-entanglement of type $\zZ{2}$.

\begin{proposition}\label{prop-3,n-entanglement}
Let $n\in\NN$ such that $n>1$ and $3\nmid n$, and let $K$ be a quadratic field contained in $\Q(\zeta_n)$. Then there are infinitely many $\overline{\QQ}$-isomorphism classes of elliptic curves with a Weil $(3,n)$-entanglement of type $\ZZ/2\ZZ$, such that the entanglement field is $K$. 
\end{proposition}

\begin{proof} Let $n>1$ and $3\nmid n$, and let $K\subseteq \Q(\zeta_n)$ be a quadratic field. Then, $K=\Q(\sqrt{d})$ for a square-free integer $d$. 
Let $t\neq 1$ be a rational number, and let $\mathcal{E}_t$ be the elliptic curve over $\QQ(t)$ given by 
\[
\mathcal{E}_t\colon y^2 = x^3 - 27t(t^3+8)x + 54(t^6-20t^3-8).
\]
The curve $\mathcal{E}_t$ is the so-called Hesse cubic, which is a model of $X(3)$ over $\Q(\sqrt{-3})$ (see \cite[Section 6.2]{GJLR}). In particular, for each $1\neq t\in\Q$, $\mathcal{E}_t$ has a rational point of order $3$ and $\Q(\mathcal{E}_t[3])=\Q(\sqrt{-3})$. Twisting $\mathcal{E}_t$ by the square-free integer $d$ results in a curve $\mathcal{E}_t^{(d)}$ 
\[
\mathcal{E}_{t}^{(d)}\colon dy^2 = x^3 - 27t(t^3+8)x + 54(t^6-20t^3-8)
\]
such that the field of definition of the point of order $3$ changing from $\QQ$ to $\QQ(\sqrt{d})=K$, and $\Q(\mathcal{E}_{t}^{(d)}[3])=\Q(\sqrt{-3},\sqrt{d})$. 
The fact that there are infinitely many $\overline{\QQ}$-isomorphism classes of elliptic curves follows from the fact that $\mathcal{E}_t$, and hence $\mathcal{E}_{t}^{(d)}$, is not isotrivial as 
\[
j(\mathcal{E}_t)=j(\mathcal{E}_{t}^{(d)}) = \frac{t^3(t+2)^3(t^2-2t+4)^3}{(t-1)^3(t^2+t+1)^3},
\]
so the $j$-invariant in the family $\mathcal{E}_{t}^{(d)}$ is non-constant. Since $3\nmid n$ and $\Q(\mathcal{E}_{t}^{(d)}[3])=\Q(\sqrt{-3},\sqrt{d})$, it follows that $\Q(\mathcal{E}_{t}^{(d)}[3])\cap \Q(\mathcal{E}_{t}^{(d)}[n])=K$. Therefore, $\mathcal{E}_{t}^{(d)}$ has a Weil $(3,n)$-entanglement of type $\ZZ/2\ZZ$, such that the entanglement field is $K$. 
\end{proof}

\subsection{Weil $(m,n)$-entanglements of $\ZZ/2\ZZ$-type}\label{subsec:explained_3,n_type2}
The techniques of Example \ref{exam:explained_3,n} and Proposition \ref{prop-3,n-entanglement} can be extended to prove the following. 

\begin{theorem}\label{thm:m,nWeilentanglements}
Let $m \in \{3,\dots, 10,12\}$, and let $n\geq 3$ be an integer such that $\gcd(m,n)\leq 2$. Then, 
\begin{enumerate}
	\item If $m\in \{3,4,5,6,7,9\}$, then there are infinitely many $\overline{\QQ}$-isomorphism classes of elliptic curves with a Weil $(m,n)$-entanglement of type $\ZZ/2\ZZ$.
	\item If $m\in \{8,10,12\}$, then there are infinitely many  $\overline{\QQ}$-isomorphism classes of elliptic curves $E$ with a Weil $(m,\gcd(4|\Delta_E|,n))$-entanglement of type $\ZZ/2\ZZ\times \ZZ/2\ZZ$.
\end{enumerate}
\end{theorem} 

\begin{proof}
Let $m \in \{3,\dots, 10,12\}$. The modular curve $X_1(m)$ is of genus $0$ with a rational point, and the 1-parameter families of elliptic curves with a rational point of order $m$ can be found in the literature (see \cite{kubert}). Note that if $E/\Q$ is an elliptic curve with a rational point of order $m$, and $\rho_{m,t}$ is the Galois representation associated to $E[m]$, then the image of $\rho_{m,t}$ is contained in the subgroup
$$\left\{\begin{pmatrix} 1& a\\ 0& b \end{pmatrix}: a\in\Z/m\Z, b\in (\Z/m\Z)^\times \right\}\subseteq \GL(2,\Z/m\Z).$$
	Moreover, the bottom right corner is given by the $m$-th cyclotomic character $\chi_m$ because the determinant of $\Im(\rho_{m,t})$ is $\chi_m$. Let $n\geq 3$ be an integer and let $d$ be a square-free integer such that $\Q(\sqrt{d})\subseteq \Q(\zeta_n)$, but $\Q(\sqrt{d})\not\subset \Q(\zeta_m)$ or, equivalently, $\gcd(n,m)\leq 2$. Below, we show infinite families with the desired property:
	\begin{enumerate}
		\item If $m\in \{3,4,5,6,7,9\}$, the curves in the family satisfy $(\Q(E[n])\cap \Q(E[m]))\cap \Q^\text{ab}=\Q(\zeta_c,\sqrt{d})$ where $c=\gcd(m,n)$.
		\item If $m\in \{8,10,12\}$, the curves in the family satisfy that $(\Q(E[n'])\cap \Q(E[m]))\cap \Q^\text{ab}$ is equal to $\Q(\zeta_c,\sqrt{\Delta_E},\sqrt{d})$ where $c=\gcd(m,n')$, and $n'=\gcd(4|\Delta_E|,n)$. 
	\end{enumerate}  
	We give a list according to the value of $m$:
	\begin{itemize}
		\item \underline{$m=3$}:~pick $\mathcal{E}_{t}^{3,(d)}\colon dy^2 = x^3 - 27t(t^3+8)x + 54(t^6-20t^3-8),$ as in Proposition \ref{prop-3,n-entanglement}, so that $\Q(\mathcal{E}_{t}^{3,(d)}[3])=\Q(\sqrt{-3},\sqrt{d})$.
		\item \underline{$m=4$}:~pick $\mathcal{E}^{4,(d)}\colon y^2+xy-(d^2-1/16)y=x^3-(d^2-1/16)x^2$ for $d\neq 0,\pm 1/4$, which parametrizes elliptic curves with $E(\Q)\cong \Z/2\Z\oplus \Z/4\Z$ (see \cite[Appendix E]{lozano_little_book}). A calculation of the $4$-th division field reveals that $\Q(E[4])=\Q(i,\sqrt{d})$. 
		\item \underline{$m=5$}:~pick $\mathcal{E}_{t}^{5,(d)}$ to be the quadratic twist by $d$ of the model of $X(5)$ defined over $\Q(\zeta_5)$ given in \cite[Section 6.4.(1)]{GJLR}. Then, $\Q(\mathcal{E}_{t}^{5,(d)}[5])=\Q(\zeta_5,\sqrt{d})$.
		\item \underline{$m=6$}:~pick $\mathcal{E}_{t}^{6,(d)}$ to be quadratic twist by $d$ of the family of elliptic curves with torsion subgroup $\Z/2\Z\oplus \Z/6\Z$, as given in  \cite[Appendix E]{lozano_little_book}. An elliptic curve in the family $\mathcal{E}_{t}^6$ has full $2$-torsion defined over $\Q$,  an isogeny graph of type $S$, and an isogeny-torsion graph of type $([2,6],[2,2],[6],[2],[6],[2],[6],[2])$ following the notation of \cite[Table 4]{chiloyan2020classification}. Upon a quadratic twist by a square-free integer $d$, an elliptic curve in the family $\mathcal{E}_{t}^{6,(d)}$ has full $2$-torsion defined over $\Q$, and an isogeny-torsion graph of type $S$ with isomorphism types $([2,2],[2,2],[2],[2],[2],[2],[2],[2])$. Note that $\mathcal{E}_{t}^{6}$ has a unique rational $3$-isogeny by its isogeny-torsion graph, so the image of $\rho_{3}$ cannot be diagonalizable. This implies that $\Q(\mathcal{E}_{t}^{6}[2])=\Q$ and $\Q(\mathcal{E}_{t}^{6}[3])$ is an $S_3$ extension with quadratic subfield equal to $\Q(\sqrt{-3})$, and hence upon a twist by a square-free $d$, we see that $\Q(\mathcal{E}_{t}^{6,(d)}[2])=\Q$ and $\Q(\mathcal{E}_{t}^{6,(d)}[3])\cap \Q^\text{ab} = \Q(\sqrt{-3},\sqrt{d})$. 
		\item \underline{$m=7$}:~pick $\mathcal{E}_{t}^{7,(d)}$ to be the quadratic twist by $d$ of the family of elliptic curves with a rational $7$-torsion point, as given in \cite[Appendix E]{lozano_little_book}. In this case, one can show that $\Q(\mathcal{E}_{t}^{7,(d)}[7])\neq \Q(\mathcal{E}_{t}^{7,(d)}[7])\cap \Q^\text{ab}= \Q(\zeta_7,\sqrt{d})$ because the mod-$7$ image has to be a subgroup conjugate to $\{[1 \ \ast\ 0 \ \ast]\}\subseteq \GL(2,\Z/7\Z)$, otherwise the image would be diagonal and the curves would have two independent $7$-isogenies, which cannot occur.
		\item \underline{$m=8$}:~let $\mathcal{E}_{t}^{2\times 8}$ be the family of elliptic curves with $\Z/2\Z\oplus \Z/8\Z$-rational torsion subgroup, as in \cite[Appendix E]{lozano_little_book}, for $t\neq 0,1,1/2$. Then, $\Q(\mathcal{E}_{t}^{2\times 8}[2])=\Q(\sqrt{8t^2-8t+1})$. Note that the image of $\rho_{8}$ is contained in $\{[1 \ \ast\ 0 \ \ast]\}\subseteq \GL(2,\Z/8\Z)$ and the bottom right corner is the full $\chi_8$ cyclotomic character. Moreover, by \cite{chiloyan2020classification}, the isogeny-torsion graph must be $T_8$ with $([2,8],[8],[8],[2,4],[4],[2,2],[2],[2])$. If the top right corner in the image of $\rho_{8}$ was $0\bmod 4$, then $\mathcal{E}_{t}^{2\times 8}$ would have two independent $4$-isogenies, but that would contradict the fact that the curve $E_1$ in a $T_8$ graph does not have two such isogenies. Thus, the image of $\rho_8$ is $\{[1 \ (2\cdot \ast)\ 0 \ \ast]\}\subseteq \GL(2,\Z/8\Z)$, and a group-theoretic calculation shows that $F=\Q(\mathcal{E}_{t}^{2\times 8}[8])\cap \Q^{\text{ab}}$ is a degree $8$, tri-quadratic extension of $\Q$. It follows that $F=\Q(\zeta_8,\sqrt{8t^2-8t+1}))=\Q(i,\sqrt{2},\sqrt{8t^2-8t+1})$ where $\Q(\sqrt{\Delta})=\Q(\sqrt{8t^2-8t+1})$. 
		
		Now let $\mathcal{E}_{t}^{2 \times 8,(d)}$ be the quadratic twist by $d$. Then, $$\Q(\mathcal{E}_{t}^{2 \times 8,(d)}[8])\cap \Q^{\text{ab}}=\Q(\zeta_8,\sqrt{8t^2-8t+1},\sqrt{d})=\Q(i,\sqrt{2},\sqrt{8t^2-8t+1},\sqrt{d}).$$
		\item \underline{$m=9$}:~pick $\mathcal{E}_{t}^9$ to be the family of elliptic curves with a rational $9$-torsion point, as in \cite[Appendix E]{lozano_little_book}. Then, the image of $\rho_{3,t}$ cannot be diagonal, because in that case $\mathcal{E}_{t}^9$ would have two independent $3$-isogenies and a rational point of order $9$, which is impossible by the classification of isogeny-torsion graphs (see \cite[Tables 1-4]{chiloyan2020classification}). In particular, the image of $\rho_{9,t}$ is the subgroup $H=\{[1\ *\ 0\ *]\}\subseteq \GL(2,\Z/9\Z)$, where the lower right corner is the $9$-th cyclotomic character $\chi_9$. In particular, $H/[H,H]$ is of order $6$ and $\Q(\mathcal{E}_{t}^9[9])\cap \Q^\text{ab} =\Q(\zeta_9)$. Now, let $\mathcal{E}_{t,d}^9$ be the quadratic twist by $d$. Then, $\Q(\mathcal{E}_{t,d}^9[9])\cap \Q^\text{ab} =\Q(\zeta_9,\sqrt{d})$.
		\item \underline{$m=10$}:~pick $\mathcal{E}_{t}^{10,(d)}$ to be the quadratic twist by $d$ of the family $\mathcal{E}_{t}^{10}$ of elliptic curves with a rational $10$-torsion point, and discriminant $\Delta_t$. The image of $\rho_{10}$ for $\mathcal{E}_{t}^{10}$ is contained in $B=\{[1 \ \ast\ 0 \ \ast]\}\subseteq \GL(2,\Z/10\Z)$ where the bottom right corner is $\chi_{10}$, the full $10$-th cyclotomic character. Moreover, the top right corner cannot be $0~\pmod 2$ because the curves in the family would have rational $\Z/2\Z\times \Z/10\Z$-torsion, and it cannot be $0\bmod 5$ because then the curves would have two independent $5$-isogenies, which cannot occur because then the isogeny graph would contain a rational $50$-isogeny. Thus, the image of $\rho_{10}$ is the full group $B$, and the image for the twists is $\{[\psi \ \ast\ 0 \ \psi^{-1}\chi_{10}]\}\subseteq \GL(2,\Z/10\Z)$, where $\psi=\psi_d$ is the quadratic character corresponding to $\sqrt{d}$. Then, 
		$$\Q(\mathcal{E}_{t}^{10,(d)}[10])\cap \Q^\text{ab} = \Q(\zeta_5,\sqrt{\Delta_t},\sqrt{d}).$$
		\item \underline{$m=12$}:~pick $\mathcal{E}_{t}^{12,(d)}$ to be the quadratic twist by $d$ of the family $\mathcal{E}_{t}^{12}$ of elliptic curves with a rational $12$-torsion point, and discriminant $\Delta_t$. By \cite{chiloyan2020classification}, the isogeny-torsion graph of $\mathcal{E}_{t}^{12}$ is $S$ of type $([2,6],[2,2],[12],[4],[6],[2],[6],[2])$, and it follows that the only prime-power isogenies of such a curve are a $3$-isogeny and two $4$-isogenies that share an initial $2$-isogeny. It follows that the image of $\rho_{12}$ is $B=\{[1 \ \ast\ 0 \ \chi_{12}]\}\subseteq \GL(2,\Z/12\Z)$ where the top right corner is full (because otherwise there would be additional $4$- or $3$-isogenies). A group-theoretic calculation shows that $F=\Q(\mathcal{E}_{t}^{12}[12])\cap \Q^{\text{ab}}$ is a degree $8$, tri-quadratic extension of $\Q$. It follows that $F=\Q(\zeta_{12},\sqrt{\Delta_t}))=\Q(i,\sqrt{-3},\sqrt{\Delta_t})$ and 
		$$\Q(\mathcal{E}_{t}^{12,(d)}[12])\cap \Q^{\text{ab}}=\Q(\zeta_{12},\sqrt{\Delta_t},\sqrt{d})=\Q(i,\sqrt{-3},\sqrt{\Delta_t},\sqrt{d}).$$
	\end{itemize}
This concludes the proof of the theorem.
\end{proof}

\subsection{Weil $(5,n)$-entanglements of $\ZZ/4\ZZ$-type}\label{subsec:explained_5,n_type4}
Continuing with our discussion on Weil entanglements, the purpose of this subsection is to study what Weil $(5,n)$-entanglements of type $\ZZ/4\ZZ$ can occur by studying division fields of elliptic curves with 5-isogenies.

Let $G = B(5) \subseteq \GL(2,\ZZ/5\ZZ)$ be the Borel subgroup, so that $G/[G,G] \simeq (\ZZ/5\ZZ)^\times \times (\ZZ/5\ZZ)^\times$. 
Determining what entanglements of $\ZZ/4\ZZ$-type can occur amounts to understanding which $\ZZ/4\ZZ$-extensions of $\QQ$ can appear as subfields of the 5-division field of an elliptic curve with a 5-isogeny. To study this, we begin by analyzing which quadratic extensions can occur as subfields.

We start by fixing a model for the generic elliptic curve with a 5-isogeny. 
Let 
\[
\mathcal{E}_t\colon y^2 = x^3 - 3^3  \frac{(t^2+10t+5)^3}{(t^2+\frac{22}{25}t+\frac{1}{5})(t^2-20t-25)^2}\cdot x + 2\cdot 3^3\frac{(t^2+10t+5)^3}{(t^2+\frac{22}{25}t+\frac{1}{5})(t^2-20t-25)^2}.
\]
Then $\mathcal{E}_t$ is an elliptic curve over $\Q(t)$ with the property that every elliptic curve $E'/\QQ$ with a 5-isogeny is isomorphic to a twist of a specialization of $\mathcal{E}_t$. 
The 5-division polynomial of $\mathcal{E}_t$ has a quadratic factor whose roots define $\QQ(x(P_t))$ where $P_t$ generates the kernel of the 5-isogeny for $\mathcal{E}_t$. Using \texttt{Magma}, we check that this splitting field of this polynomial is exactly $\QQ(\sqrt{\delta(t)})$ where 
\[\delta(t) = 5\pwr{t^2 + \frac{22}{25}t + \frac{1}{5}}.\]
Notice that 
\[\delta(t) = (m(t))^2 + (n(t))^2\,\quad \hbox{ where }m(t) = t+\frac{7}{25}\hbox{ and }n(t) = 2t+\frac{24}{25}.\]
Therefore, for any $t\in\QQ$, the quadratic field inside $\QQ(P_t)$ is of the form $\QQ(\sqrt{d})$ where $d$ can be written as the sum of two squares. (Note that this condition needs to be met, since the $\Q(\sqrt{d})=\Q(x(P_t))\subseteq \Q(P_t)$ and $\Q(P_t)$ is a cyclic quartic; see Remark \ref{rmk:C4extensions}.)

Next, we turn our attention to determine whether there is a $t\in \mQ$ such that $\QQ(\sqrt{\delta(t)}) = \QQ(\sqrt{d})$ for a given $d \in \NN$ that is square-free and can be written as the sum of two squares. 
Notice that in this case, twisting will not help us because $\QQ(x(E[n]))$ is invariant under twisting. This question can be reformulated as follows: given a square-free integer $d$ that can be written as the sum of two squares, say $d = m^2+n^2$, are there always rational points on the curve
\[C_d\colon5\left(x^2+\frac{22}{25}x+\frac{1}{5}\right) = dy^2.\]
The curve $C_d$ is a genus zero curve, and we now show that $C_d(\mQ)\neq \emptyset$ when $d$ can be written as the sum of two squares. In other words, $C_d$ is isomorphic to $\mP^1$ over $\mQ$ precisely when $d$ can be written as the sum of two squares. 

\begin{lemma}\label{lem-Cd}
The curve $C_d$ has a rational point if and only if $d$ can be written as the sum of two squares. 
\end{lemma}
\begin{proof}
We can rewrite the defining polynomial of $C_d$ as
\[C_d\colon \pwr{x+\frac{7}{25}}^2+\pwr{2x+\frac{24}{25}}^2 = dy^2.\]
Thus, one direction is clear: if $C_d(\Q)$ is non-empty, then $d$ can be written as a sum of two squares. Conversely, writing $d=m^2+n^2$ with $m\leq n$, we can solve the system
\[\begin{cases}
x+\frac{7}{25}= my,\\
2x+\frac{24}{25} =ny.
\end{cases}\]
Using the first equation to solve for $y$, plugging into the second and solving for $x$ we find a point
\[\left(\frac{7m - 24n}{25(2n-m)},-\frac{2}{5(2n-m)}\right)\]
on $C_d(\Q)$. 
\end{proof}
Thus, for every $d\geq 1$ that can be written as the sum of two squares, there is an elliptic curve $E/\QQ$ with a 5-isogeny, whose kernel is generated by $P$ such that $\QQ(x(P)) = \QQ(\sqrt{d})$.

\begin{remark}\label{rmk:C4extensions}
It is no surprise that the only quadratic extensions that occur as subfields of the division fields of these elliptic curves are those of the form $\QQ(\sqrt{d})$ where $d = m^2+n^2$ because these are \emph{exactly} the quadratic extensions that can be embedded into a cyclic extension of degree 4. In fact, given a $d = m^2+n^2$, one can write down explicitly all of the $\ZZ/4\ZZ$-extensions of $\QQ$ that contain $\QQ(\sqrt{d})$. By \cite[Theorem 2.2.5]{jensen2002generic}, all of these extensions are of the form
\[\QQ\left(\sqrt{r\left(d+m\sqrt{d}\right)}\right)\]
for $r\in \mQ^\times/(\mQ^{\times})^2$.
\end{remark}

To summarize, we have the following proposition.

\begin{prop} 
Let $K$ be a cyclic quartic number field, and let $n\geq 2$ be an integer not divisible by $5$ such that $K\subseteq \Q(\zeta_n)$. Then, there is an elliptic curve $E/\mQ$ with a Weil $(5,n)$-entanglement of type $\Z/4\Z$ such that the entanglement field is precisely $K$. In particular, there are infinitely many $\Qbar$-isomorphism classes of elliptic curves $E/\Q$ with a Weil $(5,n)$-entanglement of type $\Z/4\Z$.
\end{prop}

\begin{proof}
	Let $K$ be a cyclic quartic field, and let $F=\Q(\sqrt{d})$ be the intermediary quadratic field contained in $K$, where $d$ is a square-free integer. By our comments in Remark \ref{rmk:C4extensions}, the integer $d$ is a sum of two squares, say $d=m^2+n^2$ with $m\leq n$. Thus, Lemma \ref{lem-Cd} shows that there is a rational number $t\in\Q$ such that $\Q(\sqrt{\delta(t)})=\Q(\sqrt{d})=F$. If we let $E = \mathcal{E}_t$ be the elliptic curve over $\mQ(t)$ defined at the beginning of Section \ref{subsec:explained_5,n_type4}, then $F=\Q(x(P_t))$ where $\langle P_t \rangle$ is the kernel of a $5$-isogeny. Thus, $K' = \Q(P_t)$ is a cyclic quartic extension of $\Q$.
	
	By \cite[Theorem 2.2.5]{jensen2002generic}, there are integers $r$ and $r'$ such that 
	\[K=\QQ\left(\sqrt{r\left(d+m\sqrt{d}\right)}\right) \text{ and } K'=\QQ\left(\sqrt{r'\left(d+m\sqrt{d}\right)}\right)\]
	where $d=m^2+n^2$ as before. Let $q$ be an integer such that $q\equiv r/r' \bmod (\Q^\times)^2$. Let $E^{(q)}$ be the quadratic twist of $E$ by $q$. Then, $E^{(q)}$ also has a $\Q$-rational $5$-isogeny, with kernel $\langle P_{t,q} \rangle$ such that 
	$$\Q(x(P_{t,q}))=\Q(x(P_t))=F \text{ and } \Q(P_{t,q}) = K.$$
	It follows that $\Q(E^{(q)}[5])\cap \Q^\text{ab} = K(\zeta_5)$, and $K= K(\zeta_5)\cap \Q(\zeta_n)\subseteq \Q(\zeta_n)$ (note: we know that $\Q(\zeta_5)\cap \Q(\zeta_n)=\Q$ because $n$ is not divisible by $5$ by hypothesis). Thus, $E$ has a Weil $(5,n)$-entanglement of type $\Z/4\Z$, with entanglement field $K$  as desired.
\end{proof}

\begin{example}
Consider the case when $d=13 = 2^2 + 3^2$. Plugging in $m=2$ and $n=3$ to the parametrization from Lemma \ref{lem-Cd}, we get the point $(-27/25, 2/5)$ on 
\[
C_{13}\colon 5\left(x^2+\frac{22}{25}x+\frac{1}{5}\right) = 13y^2.
\]
Using the $x$-coordinate of this point to specialize the $j$-invariant of $\mathcal{E}_t$, we see that the minimal quadratic twist of elliptic curves in that $\Qbar$-isomorphism class is the elliptic curve with LMFDB label \href{https://www.lmfdb.org/EllipticCurve/Q/12675bn2/}{\texttt{12675.e2}}. The field of definition of the kernel of 5-isogeny of this curve is given in the LMFDB and is the field \href{https://www.lmfdb.org/NumberField/4.0.2197.1}{4.0.2197.1}, which contains $\QQ(\sqrt{13})$, as desired. 
\end{example}

\subsection{Weil $(7,n)$-entanglements of $\ZZ/6\ZZ$-type}\label{subsec:explained_7,n_type6}
We now turn our attention to studying Weil $(7,n)$-entanglements of $\ZZ/6\ZZ$-type. 
Let $\mathcal{E}_t$ be a generic elliptic curve that has a rational $7$-isogeny $\phi$, defined by
\begin{align*}
\mathcal{E}_t\colon y^2 &=  x^3 - 27(t^2+13t+49)^3(t^2+245t+2401)x\\
  &\phantom{= x^3} + 54(t^2+13t+49)^4(t^4-490t^3-21609t^2-235298t-823543). 
\end{align*}
We have found $\mathcal{E}_t$ by twisting a model of $X_0(7)$. Let $\langle P_t \rangle \subseteq \mathcal{E}_t[7]$ be the kernel of $\phi$, so $P_t$ is defined over a $\zZ{6}$-extension, namely $\mQ(P_t)/\mQ$.

Since $\mQ(P_t)/\mQ$ is a cyclic sextic, it has unique quadratic and cubic subfields that generate the whole field, and these are exactly $ \mQ(y(P_t))$ and $\mQ(x(P_t))$. 
After a suitable quadratic twist, the quadratic field $\mQ(y(P_t))$ in $\mQ(P_t)$ is arbitrary (without modifying the field of $x$-coordinates $\Q(x(P_t))$), so it suffices to determine which cubic extensions arise as $\mQ(x(P_t))$.

We compute the $7$-division polynomial of $\mathcal{E}_t$, find its lone degree $3$ factor, and compute that it defines the same field as the polynomial
\begin{align*}
g_t(x) &:= x^3 + 147(t^2 + 13t + 49)x^2 + 147(t^2 + 13t + 49)(33t^2 + 637t + 2401)x \\ &\phantom{:= x^3}+ 49(t^2 + 13t + 49)(881t^4 + 38122t^3 + 525819t^2 + 3058874t + 5764801).
\end{align*}
In particular, $g_t(x)$ defines $ \mQ(x(P_t))$. 
We now need to determine which $\zZ{3}$-extensions of $\mQ$ the polynomial $g_t(x)$ parametrizes, and somewhat surprisingly, it parametrizes all such extensions.  

\begin{lemma}
The polynomial $g_t(x)$ parametrizes all $\zZ{3}$-extensions of $\mQ$, i.e., every $\zZ{3}$-extension of $\mQ$ occurs as the Galois group of the splitting field  of $g_t(x)$, for some value of $t\in \Q$. 
\end{lemma}

\begin{proof}
By \cite[Example on page 30]{jensen2002generic}, every $\zZ{3}$-extension of $\mQ$ occurs as the Galois group of the splitting field for some specialization of polynomial $f_t(x) = x^3 - tx^2 + (t-3)x + 1$. 
We verified (using \texttt{Magma}) that $g_t(x)$ and $f_{\frac{49}{t} + 8}(x)$ define the same degree $3$ extension of $\mQ(t)$, and so it remains to treat the case of $t=0$. We complete the proof by noting that the cyclic cubic extension obtained from $f_0(x) = x^3 - 3x + 1$ is also obtained from $f_{-20}(x)$. 
\end{proof}

We have proved the following result.

\begin{prop}
	Let $K$ be a cyclic sextic number field, and let $n\geq 2$ be an integer not divisible by $7$ such that $K\subseteq \Q(\zeta_n)$. Then, there is an elliptic curve $E/\mQ$ with a Weil $(7,n)$-entanglement of type $\Z/6\Z$, such that the entanglement field is precisely $K$. In particular, there are infinitely many $\Qbar$-isomorphism classes of elliptic curves $E/\Q$ with a Weil $(7,n)$-entanglement of type $\Z/6\Z$.
\end{prop}

In Section \ref{sec:concludingremarks}, we speculate as to whether these types of results can be extended to find other families of Weil $(m,n)$-entanglements of $\zZ{a}$-type where $a$ is a positive integer greater than six.

\section{\bf CM entanglements}
\label{sec:explainedCM}
The goal of this section is to classify the image of the adelic representations attached to elliptic curves over $\mQ$ with CM by understanding the entanglements of their division fields.
To do this, we first compute the adelic image of a representative set of curves over $\QQ$ and then describe how the adelic image changes under twisting. 
Throughout this section, let $E/\QQ$ be an elliptic curve with CM by the order of an imaginary quadratic extension $K/\mQ$ of conductor $f$, which we shall denote by $\OO_{K,f}$.

To start, from \cite[Theorem 1.2]{lozanoCMreps}, if we define $\delta$ and $\phi$ as in Definition \ref{defn-CMimage} and let 
\[
\mathcal{N}_{\delta,\phi}(\widehat{\Z}) = \varprojlim_{n} \mathcal{N}_{\delta,\phi}(n),
\]
then there is a compatible system of bases for $E[n]$ such that the image of $\rho_{E}$ is contained in $\mathcal{N}_{\delta,\phi}(\widehat{\Z})$, and the index $[\mathcal{N}_{\delta,\phi}(\widehat{\Z}):\Im(\rho_{E})]$ divides $|\OO_{K,f}^\times|$. 
Before discussing the results, we highlight the difference between entanglements of division fields of elliptic curves with CM and elliptic curves without CM. 
Indeed, by \cite[Lemma 3.15]{BCSTorPointsOnCM}, we have that for $E$ with CM by $\cO_{K,f}$ and $n\geq 3$, the field $K$ is always contained in the $n$-division field $ \Q(E[n])$.
This is something that cannot happen for a curve without CM since in this case $[\GL(2,\widehat{\ZZ}): \Im(\rho_E)]$ is finite by Serre's open image theorem. 

We note that the group $\mathcal{N}_{\delta,\phi}(\widehat{\Z})$ already accounts for these entanglements in the following sense. If we put $ \cC_{\delta,\phi}(\widehat{\Z})=\varprojlim \cC_{\delta,\phi}(n)$, then for an elliptic curve over $\Q$ with CM by $K$, we have that the image of $\rho_E|_K:\Gal(\overline{\QQ}/K)\to \mathcal{N}_{\delta,\phi}(\widehat{\Z})$ is in fact contained in $ \cC_{\delta,\phi}(\widehat{\Z})$ and we still have that the index $[\mathcal{\cC}_{\delta,\phi}(\widehat{\Z}):\Im(\rho_{E}|_K)]$ divides $|\OO_{K,f}^\times|$ as before. 
Hence, the CM entanglement of Section \ref{sec-CM-entanglements} is already ``baked into'' the definition of $\mathcal{N}_{\delta,\phi}(\widehat{\Z})$, and if the index of the image in $\mathcal{N}_{\delta,\phi}(\widehat{\Z})$ is larger than $1$, then there must be an additional entanglement beyond the CM entanglement.

From the above result of the second author, we deduce one way to determine the adelic image of an elliptic curve with CM. 
If we can show that there exists an $n_0\in \NN$ such that $[\mathcal{N}_{\delta,\phi}(n_0) : \Im(\rho_{E,n_0})] = |\OO_{f,K}^\times|$, then $\Im(\rho_{E}) = \pi_{n_0}^{-1}(\Im(\rho_{E,n_0}))$ where $\pi_{n_0}\colon\mathcal{N}_{\delta,\phi}(\widehat{\mZ})\to \mathcal{N}_{\delta,\phi}(n_0)$ is the standard component wise reduction mod $n_0$ map.

Our first theorem handles the case when $E/\Q$ has CM by an order $\OO_{K,f}$ of $K = \QQ(\sqrt{-p})$ for some odd $p$ and $|\OO_{K,f}^{\times}| = 2.$

\begin{theorem}\label{prop:CM_Image}
	Let $E/\Q$ be an elliptic curve with CM by an order $\OO_{K,f}$ in an imaginary quadratic field $K$ with $\Delta_K\neq -4,-8$ and $j(E)\neq 0$. Choose compatible bases of $E[n]$, for each $n\geq 2$, such that the image of $\rho_{E}$ is contained in $\mathcal{N}_{\delta,\phi}(\widehat{\Z})$.  Then, the index of the image of $\rho_{E}$ in $\mathcal{N}_{\delta,\phi}(\widehat{\Z})$ is exactly $2$. 
\end{theorem}

\begin{proof}
Let $E/\Q$ be an elliptic curve with CM by an order $\OO_{K,f}$ in $K$  with $\Delta_K\neq -4,-8$ and $j(E) = j_{K,f}\neq 0$, let $\delta$ and $\phi$ be chosen as in Definition \ref{defn-CMimage}, and fix a system of compatible bases for $E[n]$ such that the image of $\rho_{E,n}$ is contained in $\mathcal{N}_{\delta,\phi}(n)$. 
Part (1) of \cite[Theorem 1.2]{lozanoCMreps} gives us that the index of the image of $\rho_{E,n}$ inside $\mathcal{N}_{\delta,\phi}(n)$ is a divisor of the order of $\OO_{K,f}^\times$. 
From the hypotheses, we have that $\OO_{K,f}^\times = \{\pm 1 \}$ (see, for instance, \cite[Lemma 4.2]{lozanoCMreps}), which implies that $[\mathcal{N}_{\delta,\phi}(n) :\im(\rho_{E,n})]$ is $1$ or $2$, and, in fact, we have $j_{K,f}\neq 0,1728$ and so $j_{K,f}$ is one of the $j$-invariants that appear in Table \ref{table:CMjnvariant}.

\begin{table}[h!]
	\begin{tabular}{lccccc}
		& & & & \\
		$j$-invariant & $\Delta_K$ & $f$ &  $N_E$ & LMFDB Label(s) \\
		\toprule %
		\multirow{1}*{$j = 2^4\cdot 3^3\cdot 5^3$}  & \multirow{2}*{$-3$} & $2$ & $2^2\cdot 3^2$ & \href{https://www.lmfdb.org/EllipticCurve/Q/36/a/1}{\texttt{36.a1}}, \href{https://www.lmfdb.org/EllipticCurve/Q/36/a/2}{\texttt{36.a2}} & \\
		\cmidrule(l){3-6}
		\multirow{1}*{$j = 	-2^{15}\cdot 3\cdot 5^3$}  &   & $3$ & $3^3$ & \href{https://www.lmfdb.org/EllipticCurve/Q/27/a/1}{\texttt{27.a1}},  \href{https://www.lmfdb.org/EllipticCurve/Q/27/a/2}{\texttt{27.a2}} \\
		
		\midrule
		\multirow{1}*{$j = -3^3\cdot 5^3$}  & \multirow{2}*{$-7$} & $1$ & $7^2$ & \href{https://www.lmfdb.org/EllipticCurve/Q/49/a/2}{\texttt{49.a2}}, \href{https://www.lmfdb.org/EllipticCurve/Q/49/a/4}{\texttt{49.a4}} & \\
		\cmidrule(l){3-6}
		\multirow{1}*{$j = 	3^3\cdot 5^3\cdot 17^3$}  &   & $2$ & $7^2$ & \href{https://www.lmfdb.org/EllipticCurve/Q/49/a/1}{\texttt{49.a1}},  \href{https://www.lmfdb.org/EllipticCurve/Q/49/a/3}{\texttt{49.a3}} \\
		\midrule
		\multirow{1}*{$j = -2^{15}$} & $-11$ & $1$ & $11^2$ &   \href{https://www.lmfdb.org/EllipticCurve/Q/121/b/1}{\texttt{121.b1}}, \href{https://www.lmfdb.org/EllipticCurve/Q/121/b/2}{\texttt{121.b2}} \\
		\midrule
		\multirow{1}*{$j = -2^{15}\cdot 3^3$}  & $-19$ & $1$ &  $19^2$  & \href{https://www.lmfdb.org/EllipticCurve/Q/361/a/1}{\texttt{361.a1}}, \href{https://www.lmfdb.org/EllipticCurve/Q/361/a/2}{\texttt{361.a2}} \\
		\midrule
		\multirow{1}*{$j = -2^{18} \cdot 3^3 \cdot 5^3$}  & $-43$ & $1$ &  $43^2$  & \href{https://www.lmfdb.org/EllipticCurve/Q/1849/b/1}{\texttt{1849.b1}}, \href{https://www.lmfdb.org/EllipticCurve/Q/1849/b/2}{\texttt{1849.b2}} \\
		\midrule
		\multirow{1}*{$j = -2^{15}\cdot 3^3\cdot 5^3\cdot 11^3$}  & $-67$ & $1$ &  $67^2$  & \href{https://www.lmfdb.org/EllipticCurve/Q/4489/b/1}{\texttt{4489.b1}}, \href{https://www.lmfdb.org/EllipticCurve/Q/4489/b/2}{\texttt{4489.b2}} \\
		\midrule
		\multirow{1}*{$j = -2^{18}\cdot 3^3\cdot 5^3\cdot 23^3\cdot 29^3$}  & $-163$ & $1$ &  $163^2$  & \href{https://www.lmfdb.org/EllipticCurve/Q/26569/a/1}{\texttt{26569.a1}}, \href{https://www.lmfdb.org/EllipticCurve/Q/26569/a/2}{\texttt{26569.a2}}\\
		\bottomrule
	\end{tabular}
	\vspace*{.5em}
	\caption{{{Elliptic curves with CM and non-maximal image for $p\mid\Delta_K$}}}
	\label{table:CMjnvariant}
	\end{table}

Thus, $E$ is a quadratic twist of one of the curves $E'/\Q$ whose LMFDB labels appear listed in Table \ref{table:CMjnvariant} (see Lemma 9.6 of \cite{lozano2013field}). 
Let $p>2$ the unique prime that divides $\Delta_K$. 
The elliptic curves $E'/\Q$ listed by their LMFDB label in Table \ref{table:CMjnvariant} are chosen among all those elliptic curves over $\Q$ with $j=j_{K,f}$ such that the image of the $p$-adic representation $\rho_{E',p^\infty}$ has index $2$ within $\mathcal{N}_{\delta,\phi}(p^\infty)$. 

Each curve $E'/\QQ$ listed in Table \ref{table:CMjnvariant} has a vertical collapsing that can be explained by one of two possible phenomena. 
In all cases, $E'$ has a $p$-isogeny (where again $p$ is the unique odd prime dividing $\Delta_K$). 
Let $M/\QQ$ be the field of definition of a generator of the kernel of the $p$-isogeny with domain $E$. 
Then the first possibility is that $[M:\QQ] = (p-1)/2$, which 
forces the image of $\rho_{E,p}$ to be strictly smaller than $\mathcal{N}_{\delta,\phi}(p)$.
The other possibility is that $[M:\QQ] = (p-1)$, but the unique quadratic subfield of $M$ coincides with the unique quadratic subfield of $\QQ(\zeta_p)$.
Since $\QQ(\zeta_p)$ is always inside $\QQ(E'[p])$ by the Weil pairing, this coincidence forces the image of $\rho_{E,p}$ to be smaller than $\mathcal{N}_{\delta,\phi}(p)$. 

Let $d\in \ZZ$ be the unique square-free integer not divisible by $p$ such that the quadratic twist of $E$, $E^{(d)}$, is $\QQ$-isomorphic to one of the curves listed in Table \ref{table:CMjnvariant}, call it $E'$. 
We can choose $d$ not divisible by $p$ since the pairs of curves listed are quadratic twists of each other by $-p$. 
If $d = 1$, then we are done. Otherwise, we have that $\QQ(E[p]) = \QQ(E'[p],\sqrt{d})$ and $\rho_{E,p}$ surjects onto $\mathcal{N}_{\delta,\phi}(p).$ 
Next we let $n$ be the conductor of the extension $\QQ(\sqrt{d})/\QQ$ and note that since $d$ is not divisible by $p\geq 3$, we have that $n$ is not divisible by $p$, and $K\neq \Q(\sqrt{d})$. 
Thus, we have that $\gcd(n,p) = 1$ and by the Weil pairing and definition of conductor, $\QQ(\sqrt{d}) \subseteq \QQ(E[p]) \cap \QQ(E[n])$. Since $\QQ(\sqrt{d})\neq K$, this is not a CM entanglement as in Definition \ref{defn:CM-entanglement}. 
This entanglement together with \cite[Theorem 1.2]{lozanoCMreps} gives us that $1< [\mathcal{N}_{\delta,\phi}(pn) : \Im(\rho_{E,pn})]\leq 2$. Thus, $[\mathcal{N}_{\delta,\phi}(pn) : \Im(\rho_{E,pn})] = 2$ and $\Im(\rho_{E}) = \pi_{pn}^{-1}(\Im(\rho_{E,pn}))$ where $\pi_{pn}\colon\mathcal{N}_{\delta,\phi}(\widehat{\Z})\to \mathcal{N}_{\delta,\phi}(pn)$ is the standard component wise reduction mod $pn$ map. 
\end{proof}

\begin{prop}
	Let $E/\QQ$ be an elliptic curve with CM by an order $\OO_{K,f}$ where $K = \QQ(\sqrt{-2})$. Then, the index of the image of $\rho_{E}$ in $\mathcal{N}_{\delta,\phi}(\widehat{\Z})$ is $2$ .
\end{prop}

\begin{proof}
Let $E/\QQ$ be an elliptic curve with CM by an order $\OO_{K,f}$ where $K = \QQ(\sqrt{-2})$, and define auxiliary curves: 
\begin{align*}
 E_1\colon y^2 &= x^3+x^2-3x+1,\ \ \ \ \ \ \ \ \ \ \
 E_2\colon y^2 = x^3-x^2-13x+21,\\
 E_3\colon y^2 &= x^3+x^2-13x-21,\hbox{\ \ \ \ \ \ \ \ }
 E_4\colon y^2 = x^3-x^2-3x-1.
 \end{align*}
Then $E_1,\dots ,E_4$ all have CM by $\ZZ[\sqrt{-2}]$ and form a complete set of quadratic twists by $-1 ,2,$ and $-2$. That is, they fit into the following diagram.  
$$\xymatrix{
 E_1 \ar@{<->}[rr]^{\hbox{Twist by 2}} \ar@{<->}[dd]_{\hbox{Twist by $-1$}} & & E_2  \ar@{<->}[dd]^{\hbox{Twist by $-1$}}\\
& & \\
 E_4 \ar@{<->}[rr]_{\hbox{Twist by 2}}& &E_3 
}$$
Because of this, there exists a unique square free, odd, positive integer $d$, such that $E$ is the quadratic twist by $d$ of exactly one of $E_1$, $E_2$, $E_3$ or $E_4$. We write $E=E_i^d$ for some $i=1,\ldots,4$. Using \texttt{Magma}, for each $i = 1,\dots,4$, we compute that $[\QQ(E_i[8]):\QQ] = 32$ and $\mathcal{N}_{\delta,\phi}(8)$ has size 64. 
%
%
Therefore, for each of these curves $[\mathcal{N}_{\delta,\phi}(8):\Im(\rho_{E,8})] = 2$, which forces $[\mathcal{N}_{\delta,\phi}:\Im(\rho_E)] = 2$ as well. 

Letting $F$ be the splitting field of $f(x) = x^8+6x^4+1$ and 
\[L_d = \QQ\left(\sqrt{d(2+\sqrt{2})}\right),\] 
a computation of the $4$-division field of $E_i$ shows that $\QQ(E_i[4]) = FL_1$, for $i=1,2,3,4$, and the $4$-division field of the twist is $\QQ(E[4])=FL_d$. Noticing that both $\QQ(E[4])$ and $\QQ(\zeta_{16}) = \QQ(\sqrt{-1}, \sqrt{2+\sqrt{2}})$ are contained in $\QQ(E[16])$ gives us that 
\[
\QQ(\sqrt{d}) \subseteq \QQ\left(\sqrt{d},\sqrt{2+\sqrt{2}}\right ) = \QQ\left(\sqrt{d(2+\sqrt{2})},\sqrt{2+\sqrt{2}}  \right)\subseteq \QQ(E[16]).
\]
%
%
%

Moreover, since $d$ is assumed to be odd and positive, the conductor of $L_d$ is exactly 16 times the conductor of $\QQ(\sqrt{d})$. Indeed, the field $L_d$ is contained inside $\QQ(\sqrt{d})\QQ(\zeta_{16})^+$, and so if $\QQ(\sqrt{d})/\QQ$ is of conductor $n$, then $L_d\subseteq\QQ(\zeta_{16n})$  and $L_d$ is not contained in any smaller cyclotomic field. We also note that since $d$ is positive, $\QQ(\sqrt{-2}) \neq \QQ(\sqrt{d})$ and so any entanglement that involves $\QQ(\sqrt{d})$ is not one of the CM entanglements already accounted for in $\mathcal{N}_{\delta,\phi}(\widehat{\mZ})$. These two things force the mod $16n$ image to be index 2 inside of $\mathcal{N}_{\delta,\phi}(16n)$ and the $\Im(\rho_{E}) = \pi^{-1}_{16n}(\Im(\rho_{E,16n}))$. 
\end{proof} 

\begin{prop}\label{thm-j1728-index2}
	Let $E/\Q$ be an elliptic curve with $j(E) =1728$ and choose compatible bases of $E[n]$ for each $n\geq 2$ such that the image of $\rho_{E}$ is contained in $\mathcal{N}_{\delta,\phi}(\widehat{\Z})$. Then, the index of the image of $\rho_{E}$ in $\mathcal{N}_{\delta,\phi}(\widehat{\Z})$ is $2$ or $4$.
\end{prop}

\begin{proof}
	Let $E/\Q$ be an elliptic curve with $j_{K,f}=1728$, given by a Weierstrass model $y^2=x^3+sx$, for some $s\in\Q$. 
	We choose compatible bases of $E[n]$ for each $n\geq 2$ such that the image of $\rho_{E}$ is contained in $\mathcal{N}_{\delta,\phi}(\widehat{\Z})$.	
	By \cite[Theorem 1.2]{lozanoCMreps}, the image of $\rho_{E}$ has index dividing $\# (\mathcal{O}_{K,f}^\times) = 4$, since $K=\Q(i)$ and $f=1$ for $j_{K,f}=1728$.

	Suppose towards a contradiction that the index is $1$. 
	Then, the image of $\rho_{E,2^\infty}$ must have index $1$ in $\mathcal{N}_{\delta,\phi}(2^\infty)$. 
	By the proof of \cite[Theorem 9.5]{lozanoCMreps} which contains a description of the possible $2$-adic images when $j_{K,f}=1728$, we have that $s\notin \pm (\Q^\times)^2$ nor $\pm 2(\Q^\times)^2$, and $\Q(E[4])=\Q(i,\sqrt{2},\sqrt[4]{-s})$. 
	If we write $s=2^{e_1}\cdot m\cdot n^2$ with $m$ an odd square-free integer, then $\sqrt{m}$ and $\sqrt{-m}$ belong to $\Q(E[4])$. 
	Now consider $\rho_{E,m}\colon \GQ \to \GL(2,\Z/m\Z)$. Since $\det(\rho_{E,m})$ is the $m$-th cyclotomic character, then one of $\sqrt{\pm m}$ is contained in $\Q(E[m])$. 
	Therefore $\Q(i,\sqrt{m})/\Q(i)$ is a non-trivial quadratic extension, and $\Q(i,\sqrt{m})\subseteq \Q(i,E[4])\cap \Q(i,E[m])$. 
	It follows that the image of $\rho_{E,4m}\colon \Gal(\overline{\Q(i)}/\Q(i)) \to \GL(2,\Z/4m\Z)$ is not surjective onto $\cC_{-1,0}(4m)\cong \cC_{-1,0}(4)\times \cC_{-1,0}(m)$, and therefore the image of $\rho_{E,4m}$ in $\mathcal{N}_{-1,0}(4m)$ is of index at least $2$, which is a contradiction to our initial assumption.
\end{proof}

\begin{prop}\label{thm-j0-index2}
	Let $E/\Q$ be an elliptic curve with $j_{K,f}=0$, and choose compatible bases of $E[n]$ for each $n\geq 2$ such that the image of $\rho_{E}$ is contained in $\mathcal{N}_{\delta,\phi}(\widehat{\Z})$. Then, the index of the image of $\rho_{E}$ in $\mathcal{N}_{\delta,\phi}(\widehat{\Z})$ is $2$ or $6$.
\end{prop}

\begin{proof}
	Let $E/\Q$ be an elliptic curve with $j_{K,f}=0$, given by a Weierstrass model $y^2=x^3+s$, for some $s\in\Q^\times$. 	
	We choose compatible bases of $E[n]$ for each $n\geq 2$ such that the image of $\rho_{E}$ is contained in $\mathcal{N}_{\delta,\phi}(\widehat{\Z})$.	
	By \cite[Theorem 1.2]{lozanoCMreps}, the image of $\rho_{E}$ has index dividing $\# (\mathcal{O}_{K,f}^\times) = 6$, since $K=\Q(\zeta_3)$ and $f=1$ for $j_{K,f}=0$. 
	A straight forward computation shows that in this case $\QQ(E[3]) = \QQ(\sqrt{-3},\sqrt{s},\sqrt[3]{4s})$. 
	So if $s = t^2$ or $-3t^2$ for some $t\in \QQ^\times$, then $[\mathcal{N}_{\delta,\phi}(3) : \Im(\rho_{E,3})]$ is divisible by 2 completing this case. 
	Otherwise, we have that $\QQ(\sqrt{s}) \subseteq \QQ(E[3])$ and $n$, the conductor of $\QQ(\sqrt{s})/\QQ$, is not $3$, and $\QQ(\sqrt{s})\neq K = \QQ(\zeta_3)$ (thus the entanglement described here is not a CM entanglement). 
	This forces $E$ to have an unexpected quadratic intersection between its $3$ and $n$ division fields forcing $[ \mathcal{N}_{\delta,\phi}(\widehat{\Z}): \Im(\rho_{E}) ]$ to be even. 
\end{proof}

\section{\bf Entanglements for higher dimensional abelian varieties}
\label{sec:explainedhigherdim}
In this final section, we investigate entanglements for higher dimensional abelian varieties. 

\subsection{Galois representations attached to principally polarized abelian varieties}\label{subsec:GaloisrepPPAV}
Before getting to entanglements, we describe the mod $m$ image of Galois for a principally polarized abelian variety $A/\mQ$ of dimension $g\geq 1$. 
For each positive integer $m$, denote by $A[m]$ the kernel of the multiplication-by-$m$ map $[m]\colon A\to A$. 
We have that $A[m]$ is a finite group scheme over $\mQ$ and that $A[m](\overline{\mQ})$ forms a finite subgroup of $A(\overline{\mQ})$, called the \cdef{$m$-torsion subgroup}, which is isomorphic to $(\zZ{m})^{2g}$. 
$A[m]$ is a finite \'etale $\mQ$-group scheme, which we will identify with the $\Aut(\overline{\mQ}/\mQ)$-moudle $A[m](\overline{\mQ})$.  

As with elliptic curves, we can make the following definition of an entanglement. We note that this definition makes sense for any field of characteristic zero $K$, but we restrict to the case of $K = \mQ$ for simplicity. 

\begin{defn}
Let $A/\QQ$ be a principally polarized abelian variety, and let $a,b$ be positive integers. We say that $A$ has an \cdef{$(a,b)$-entanglement} if 
\[
K = \QQ(A[a]) \cap \QQ(A[b])\neq \QQ(A[d]),
\] 
where $d=\gcd(a,b)$. 
The \cdef{type} $T$ of the entanglement is the isomorphism class of $\Gal(K/\QQ(A[d]))$.
\end{defn}

In \cite{danielsMorrow:Groupentanglements}, the first and last author provided a group theoretic definition of entanglements for elliptic curves. 
To describe a similar construction for principally polarized abelian varieties $A/\mQ$, we recall the mod $m$ Galois representation attached to $A$.

The absolute Galois group $G_{\mQ}$ acts on $A(\overline{\mQ})$, and since this action is compatible with the group law on $A(\overline{\mQ})$, it gives rise to an action of $G_{\mQ}$ on the $m$-torsion subgroup $A[m]$ for each positive integer $m$. 
This action produces a Galois representation 
\[
G_{\mQ} \to \Aut(A[m]) \simeq \GL(2g,\zZ{m}).
\]

There is a constraint on the image of this above representation coming from the Weil pairing. 
More precisely, for each $m$, let $\mu_m$ be the group of $m$-th roots of unit in $\overline{\mQ}$ and let $\chi_m\colon G_{\mQ}\to \mu_m$ denote the mod $m$ cyclotomic character. 
By composing the Weil pairing with the principal polarization isomorphism $A\to A^{\vee}$, we have an alternating non-degenerate bilinear form
\[
e_m\colon A[m] \times A[m] \to \mu_m.
\]
Since the Weil pairing is $G_{\mQ}$-equivariant, the image of the above Galois representation is constrained to lie in $\GSp(2g,\zZ{m})$ i.e., the group of $2g\times 2g$ general symplectic matrices with coefficients in $\zZ{m}$. 
Therefore, we obtain the \cdef{mod $m$ Galois representation}
\[
\rho_{A,m}\colon G_{\mQ}\to \GSp(2g,\zZ{m}). 
\]

\subsection{Group theoretic definition of entanglements for principally polarzied abelian varieties}
With this construction, we can provide similar group theoretic definitions of explained and unexplained entanglements for principally polarized abelian varieties $A/\mQ$ as follows. 
In the case of elliptic curves over $\mQ$, the determinant of the image of the mod $m$ representation is the mod $m$ cyclotomic character, and surjectivity of this map allowed us to define the notion of explained and unexplained entanglements.

For abelian varieties $A/\mQ$ of dimension greater than one, the determinant of the image of mod $m$ representation does not always correspond to the mod $m$ cyclotomic character; we refer the reader to \cite[Lemma 4.5.1]{ribet1976galois} for a discussion of when this holds. 
As such, we cannot directly transport the definitions of explained and unexplained entanglements from elliptic curves to arbitrary principally polarized abelian varieties $A/\mQ$. 

That said, if we replace the determinant with the similitude character, then we can obtain our desired result. 
Let $\nu\colon\GSp(2g,\zZ{m}) \to \mu_m$ denote the similitude character. 
Note that the composition
\[
\begin{tikzcd}
G_{\mQ} \arrow{r}{\rho_{A,m}} & \GSp(2g,\zZ{m}) \arrow{r}{\nu} & \mu_m
\end{tikzcd}
\]
is the mod $m$ cyclotomic character $\chi_{m}$, and since $A$ is defined over $\mQ$, the map $\chi_{m}$ is surjective (see e.g., \cite[Lemma 2.1]{lombardo:Explicitsurjectivity}). 
Recall that Remark \ref{rem:Weilexplained} provides an equivalence between explained and Weil entanglements for elliptic curves, and so the above description allows us to make group theoretic definitions of Weil and non-Weil entanglements for principally polarized abelian varieties $A/\mQ$ as in \cite[Definitions 3.7 \& 3.9]{danielsMorrow:Groupentanglements} where we replace the map $\det(\cdot)$ with the similitude character $\nu(\cdot)$. 
While these group theoretic considerations take some care, the non-group theoretic definitions of abelian and Weil entanglements (Definitions \ref{defn:abelian-type} and \ref{defn:weil-type}) can be carried over \textit{mutatis mutandis}. 
The definition of CM entanglements (Definition \ref{defn:CM-entanglement}) is more subtle. The main issue is we do not have a complete description of the geometric endomorphism ring of an arbitrary principally polarized abelian variety; however, we do have such a description for abelian surfaces, the so-called Albert's classification (see e.g., \cite[p.~203]{mumford1970abelian}).

\subsection{Weil entanglements for higher dimensional abelian varieties}\label{subsec:explainedPPAV}
In \cite[Proposition 8.4]{morrow2017composite} and \cite[Theorem 1.7(2)]{danielsLR:coincidences}, we defined an explicit family of elliptic curves $E/\mQ$ which have an Weil $(2,\ell)$-entanglement of type $\zZ{3}$ where $\ell$ is some prime such that $3\mid \ell-1$. 
We generalize this construction to principally polarized abelian varieties of dimension $g$ as follows.

\begin{theorem}
Let $\ell > 3$ be a prime number. 
Suppose that $\ell - 1 = 2e$ where $e = 2g+1$ is some odd integer. 
There exist infinitely many principally polarized abelian varieties $A/\mQ$ of dimension $g$ which have a Weil $(2,\ell)$-entanglement of type $\zZ{e}$.
\end{theorem}

\begin{proof}
Our assumptions on the divisibility of $\ell - 1$ ensure that the cyclotomic field $\mQ(\zeta_{\ell})$ contains a unique subfield $K/\mQ$ of degree $e$ with Galois group $\zZ{e}$. 
By work of Gauss \cite[art.~337]{gauss1966disquisitiones}, the minimal polynomial of $K$ is of the form
\[
f(x) = x^e + a_1x^{e-1} + a_2x^{e-2} + \cdots + a_{e-1}x + a_e
\]
with coefficients 
\[
a_r = (-1)^{[r/2]} {(\ell - 1)/2 - [(r+1)/2] \choose [r/2]} \quad (0\leq r \leq e), 
\]
where $[\cdot]$ denotes the greatest integer function. 

Consider the hyperelliptic curve $X$ of genus $g$ with affine equation $y^2 = f(x)$. 
Note that since $f(x)$ is odd, we have that $X$ has a rational Weierstrass point $\infty$, and in particular, we know that $X(\mQ)\neq \emptyset$. 
Let $J_X$ denote the Jacobian of $X$, which is a principally polarized abelian variety over $\mQ$ of dimension $g$. 
As with elliptic curves, the $2$-division field $\Q(J_X[2])$ of $J_X$ is given by the splitting field of $f(x)$.
Furthermore, the Weil pairing tells us that $\mQ(\zeta_{\ell})$ is contained in the $\ell$-division field of $J_X$, and thus, $J_X$ has a Weil $(2,\ell)$-entanglement of type $\zZ{e}$. 

To complete the proof, we need to find infinitely many other hyperelliptic curves $X'/\mQ$ which have isomorphic $2$-torsion. 
To do so, we follow the construction from \cite[Section 4.4]{calegariCR:abeliansurfacesfixed2torsion}. 
Let $\alpha_1$ denote the companion matrix of $f(x)$. 
For $j = 1,\dots,e$, let $\alpha_j = \alpha_1^j - k_jI $ where $k_j$ is chosen to make $\alpha_j$ traceless. 
For indeterminates $b_1,\dots, b_e$, the curve 
\[
X(b_1,\dots,b_e)\colon y^2 = \det(xI - b_1\alpha_1 - b_2\alpha_2 - … - b_{e}\alpha_{e}) \subseteq \mA^2_{\mQ(b_1,\dots,b_e)}
\]
will have the same 2-torsion as the original curve $X$. 
After removing the discriminant locus (i.e., the values of $(b_1,\dots,b_e)$ where the above curve is singular), specializations of the curve $X(b_1,\dots,b_e)$ will produce infinitely many hyperelliptic curves $X'/\mQ$ with isomorphic $2$-torsion as $X$.
The Jacobians $J_{X'}$ of these infinitely many hyperelliptic curves $X'$ will yield infinitely many principally polarized abelian varieties $A/\mQ$ of dimension $g$ which have a Weil $(2,\ell)$-entanglement of type $\zZ{e}$.
\end{proof}

\begin{remark}
To conclude, we remark on the difficulty of studying other explained (or even Weil) entanglements for principally polarized abelian varieties not involving the 2-torsion. 
Above in Section \ref{sec:explainedentanglements}, there were two crucial aspects to our approach. 
First, we could construct an elliptic surface $\mathcal{E}/\mQ(t)$ whose specializations correspond to elliptic curves with a prescribed mod $m$ image of Galois; in particular, the moduli space of elliptic curves with a prescribed mod $m$ image of Galois was rational. 
Second, we were able to use the well established theory of elliptic division polynomials to study $m$-division fields. 
Both aspects become more complicated in the higher dimensional setting.

For the first part, we refer the reader to \cite{calegariCR:abeliansurfacesfixed2torsion} for a discussion on the rationality of moduli spaces of principally polarized abelian varieties with a prescribed mod $m$ image of Galois. 
We note that in Theorem 2 of \textit{loc.~cit.~}the authors provide for a given a genus two curve $X/\mQ$ with affine equation $y^2 = x^5 +ax^3 +bx^2 +cx+d$ an explicit parametrization of all other such curves $Y$ with a specified symplectic isomorphism $J_X[3] \simeq J_Y[3]$. 
While the parametrization is explicit, working directly with universal Weierstrass curve associated to this parametrization is difficult due to the complexity of the coefficients.

As for the second part, Cantor \cite{cantor1994analogue} has developed a theory of division polynomials for hyperelliptic curves of genus $g\geq 1$, where in the $g = 1$ setting one recovers the classical division polynomials. 
The degrees of theses polynomials become quite large as the genus grows, so understanding their Galois theory for arbitrary $g$ is difficult. 
However, it would be quite interesting to write computer program which inputs a separable polynomial $f(x)$ and an integer $m$ and outputs the $m$th division polynomial of the hyperelliptic curve with affine equation $y^2 = f(x)$.

Combining the results from \cite{calegariCR:abeliansurfacesfixed2torsion} with such a computer program, one could hope to study $(3,m)$-entanglements of genus $2$ hyperelliptic curves over $\mQ$. 
\end{remark}

\section{\bf Concluding remarks}
\label{sec:concludingremarks}
In this final section, we discuss future avenues of study for explained entanglements. 

\subsection{Recollections on generic polynomials}
In Section \ref{sec:explainedentanglements}, we saw that questions on Weil entanglements are related to questions about specializations of polynomials $P_{\textbf{t}}(X)$ in $\mQ(\textbf{t})[X]$ with $\textbf{t} = (t_1,\dots, t_n)$ parametrizing certain Galois extensions of $\mQ$. 
This latter question is related to the embedding problem (which is a generalization of the inverse Galois problem) and to the existence of generic polynomials and the generic dimension of finite groups. 

Below, we recall the definition of generic polynomials and generic dimension of a finite group. 

\begin{defn}[\protect{\cite[Definition 0.1.1]{jensen2002generic}}]\label{def:genericpoly}
Let $P_{\textbf{t}}(X)$ be a monic polynomial in $\mQ(\textbf{t})[X]$ with $\textbf{t} = (t_1,\dots, t_n)$ and $X$ being indeterminates, and let $\mathbb{L}$ be the splitting field of $P_{\textbf{t}}(X)$ over $\mQ(\textbf{t})$. Suppose that
\begin{itemize}
\item $\mathbb{L}/\mQ(\textbf{t})$ is Galois with Galois group $G$, and that 
\item every $L/\mQ$ with Galois group $G$ is the splitting field of a polynomial $P_{\textbf{a}}(X)$ for some $\textbf{a} = (a_1,\dots, a_n) \in \mQ^n$.
\end{itemize}
We say that $P_{\textbf{t}}(X)$ \cdef{parametrizes} $G$-extensions of $\mQ$ and $P_{\textbf{t}}(X)$ is a \cdef{parametric polynomial}.

The parametric polynomial $P_{\textbf{t}}(X)$ is \cdef{generic} if
\begin{itemize}
\item $P_{\textbf{t}}(X)$ is parametric for $G$-extensions over any field containing $\mQ$.
\end{itemize}
\end{defn}

\begin{defn}[\protect{\cite[Definition 8.5.1]{jensen2002generic}}]
For a finite group $G$, the \cdef{generic dimension} for $G$ over $\mQ$, written $\gd_{\mQ}G$, is the minimal number of parameters in a generic polynomial for $G$ over $\mQ$, or $\infty$ is no generic polynomial exists. 
\end{defn}

\begin{remark}
The generic dimension of $G$ is bounded below by (and when finite, is conjecturally equal to) essential dimension of $G$ (cf.~\cite[Proposition 8.5.2, Conjecture on page 202]{jensen2002generic}).
\end{remark}

We recall two results, which will help us formulate our precise question relating generic dimension of finite groups to Weil entanglements. 

\begin{theorem}[\protect{\cite[Proposition 8.5.5 \& Theorem 8.5.8]{jensen2002generic}}]\label{prop:directproductgeneric}
\begin{enumerate}
\item []
\item Let $G$ and $H$ be finite groups. Then
\[
\gd_{\mQ}G,\, \gd_{\mQ}H \leq \gd_{\mQ}(G\times H) \leq \gd_{\mQ}G \, + \, \gd_{\mQ}H.
\]
\item The only groups of generic dimension 1 over $\mQ$ are $\brk{I}, \zZ{2}, \zZ{3}, \text{ and }S_3.$ 
\end{enumerate}
\end{theorem}

\subsection{Two questions on Weil entanglements and generic polynomials}
To conclude, we propose two questions concerning the further study of Weil entanglements.

In Subsections \ref{subsec:explained_2,n_type3}, \ref{subsec:explained_3,n_type2}, \ref{subsec:explained_5,n_type4}, and \ref{subsec:explained_7,n_type6}, we encountered several examples of generic polynomials for $\zZ{2}$-extensions and $\zZ{3}$-extensions over $\mQ$, and the results from these sections lead us to ask the following question.

\begin{question}\label{question:explained}
Fix an integer $n$. 
Let $\mathcal{E}_t$ be a family of elliptic curves over $\mQ$. 
Suppose that $ \mQ(x(\mathcal{E}_t[n]))$ contains a quadratic  (resp.~cubic abelian) subfield $K_t $, which is not contained in $\mQ(\zeta_n)$. 
Let $P_t(X)$ denote the quadratic (resp.~cubic) factor of the $n$-division polynomial of $\mathcal{E}_t$. 
\begin{enumerate}
\item Is $P_t(X)$ a generic polynomial for $\zZ{2}$-extensions (resp.~$\zZ{3}$-extensions) of $\mQ$?
\item If $P_t(X)$ is not a generic polynomial for $\zZ{2}$-extensions (resp.~$\zZ{3}$-extensions) of $\mQ$, then which quadratic (resp.~cubic) number fields are realized as Galois groups of specializations of $P_t(X)$? 
\end{enumerate}
\end{question}

The reason we restrict to $K_t /\mQ(t)$ being quadratic or abelian cubic is due to Theorem \ref{prop:directproductgeneric}. 
More precisely, we observed in the above discussion that we essentially have two parameters when studying Weil entanglements:~ one of them coming from the division polynomials (i.e., the $x$-coordinate of the torsion points) and the second coming from twisting (i.e., the $y$-coordinate of the torsion points). 
As the field of definition of a point of order $N$ is generated by the $x$ and $y$ coordinates of the torsion points and twisting will only contribute a $\zZ{2}$-factor, the parameter corresponding to the $x$-coordinates must contribute a $\zZ{2}$ or $\zZ{3}$ factor. 

%

\begin{remark}
If a classification of finite groups with generic dimension $2$ over $\mQ$ existed, then we could make Question \ref{question:explained} more precise. However, no such classification exists; we do not even have a classification of finite groups with essential dimension 2 over $\mQ$ (cf.~\cite[Remark 1.4]{duncan2013essentialdimension2}). 
\end{remark}

The above question asked if one can parametrize the quadratic or cubic abelian extensions which are not contained in certain cyclotomic field. To conclude the section, we discuss how one could systematically study elliptic curves over $\QQ$ with a Weil entanglement. 
First note that a consequence of \cite{brau3in2} is that if $E/\QQ$ has surjective $\Im(\rho_{E,p^\infty})$ for all $p$, then the only possible entanglements are Serre entanglements or a $(2,3)$-entanglement of $S_3$-type coming from an inclusion $\QQ(E[2])\subseteq\QQ(E[3])$. 
Thus, to have a different Weil entanglement, one needs some exceptional prime power. 

For a given prime $p$ and natural number $n$, we consider the groups $G\subseteq \GL(2,\ZZ/p^n\ZZ)$ such that there exists an elliptic curve $E_0/\QQ$ with $\Im (\rho_{E_0,p^n})$ conjugate to $G$. 
If $[G: [G,G]]$ is strictly larger than $\varphi(p^n) = p^{n-1}(p-1)$, then this means that any elliptic curve $E/\QQ$ with $\Im (\rho_{E,p^n})$ conjugate to $G$ has $\QQ(E[p^n]) \cap \QQ^{\ab}$ which is strictly larger than $\QQ(\zeta_{p^n})$. 
For any particular elliptic curve $E/\QQ$ with mod $p^n$ image conjugate to $G$, $(\QQ(E[p^n]) \cap \QQ^{\ab})/\QQ$ is an abelian extension and thus has some conductor, say $f(\QQ(E[p^n]) \cap \QQ^{\ab})/\QQ)) = N$.  
Since $\QQ(\zeta_{p^n})\subseteq\QQ(E[p^n]) \cap \QQ^{\ab}$, we know that $N = p^n\cdot m$ for some $m$, and this $m$ tells us about the level of any explained entanglement that involves the $p^n$-division field of $E$. If $p\nmid m$ then we know that $E/\QQ$ has a Weil $(p^n,m)$-entanglement and if $p\mid m$, then $E$ has an Weil $(p^n,N)$-entanglement. 

For an example, when $p\mid m$, see \cite[Example 3.9]{danielsLR:coincidences}. 
In that example, it is shown that the elliptic curve with LMFDB label \href{https://www.lmfdb.org/EllipticCurve/Q/32a3/}{\texttt{32a3}} has the property that $\QQ(\zeta_{2^{k+1}}) \subseteq \QQ(E[2^k])$ for all $k>1$, and so the conductor of $\QQ(E[2^k])$ is divisible by $2^{k+1}$ for all $k>1$.
In both cases, the type of the Weil entanglement is the Galois group $\Gal( (\QQ(E[p^n]) \cap\QQ^{\ab})/\QQ(\zeta_{p^n}) ).$ 

With this, we see one approach to the study of explained entanglements of elliptic curves over $\QQ$ through the conductor $f((\QQ(E[p^n]) \cap \QQ^{\ab})/\QQ))$ of the extension $(\QQ(E[p^n]) \cap \QQ^{\ab})/\QQ)$. 
More precisely, we propose the following question. 

\begin{question}\label{question:boundsconductor}
For a given prime $p$ and natural number $n$, let $-G\subseteq \GL(2,\ZZ/p^n\ZZ)$ such that $-I \in G$ and there exists an elliptic curve $E/\QQ$ with $\Im(\rho_{E,p^n})$ conjugate to $G$. 
Consider the function
\[
N_{p^n,G}(X) = \# \brk{E/\QQ : \text{$\Im(\rho_{E,p^n})$ is conjugate to $G$ and } f((\mQ(E[p^n]) \cap \mQ^{\ab})/\mQ) < X}
\]
where the count is up to $\overline{\mQ}$-isomorphism. 
For which $p,n,G,$ and $X$, is the function $N_{p^n,G}(X) $ finite?
\end{question}

The work of \cite{sutherlandzywina2017primepower} makes this question concrete when restricting to group of genus 0 and level $p^n$ containing $-I$. 
For each of these groups there is an explicit parametrization for the moduli space $X_G$, and one could use this parametrization to find a 1-parameter family of number fields $K_t/\QQ$ corresponding to the family $(\QQ(E[p^n])\cap\QQ^{\ab})/\QQ$ as $E$ runs over $X_G(\QQ)$, much in the same way as we did above.

\bibliography{refs}{}
\bibliographystyle{amsalpha}
\end{document}